\documentclass[12pt,reqno,a4paper]{amsart}
\usepackage{amssymb,mathtools,extarrows}
\allowdisplaybreaks[4]
\usepackage[all]{xy}
\usepackage[margin=1in]{geometry}
\usepackage{enumitem,microtype}
\usepackage[colorlinks=true,linkcolor=blue,citecolor=blue,urlcolor=blue]{hyperref}
\hypersetup{pdftitle={Locally enhanced nontopological triangulated categories},
  pdfauthor={Chencheng Zhang},
  pdfkeywords={triangulated categories, enhancements, A-infinity categories, Toda brackets, category cohomology, filtered colimits}}
\setlist{nosep,leftmargin=2.2em}
\setlist[enumerate]{label=\textup{(\arabic*)},ref=\arabic*}
\makeatletter
\renewcommand{\paragraph}{\@startsection{paragraph}{4}%
  \z@\z@{-\fontdimen2\font}%
  {\normalfont\bfseries\itshape}}
\makeatother

\theoremstyle{plain}
\newtheorem{lemma}{Lemma}[section]
\newtheorem{proposition}[lemma]{Proposition}

\newtheorem{corollary}[lemma]{Corollary}
\theoremstyle{remark}
\newtheorem{remark}[lemma]{Remark}
\theoremstyle{definition}
\newtheorem{definition}[lemma]{Definition}
\theoremstyle{plain}
\newtheorem*{maintheorem}{Main Theorem}

\title[A nontopological triangulated category with is \texorpdfstring{$\theta$}{theta}-dg enhanced]{A nontopological triangulated category\\ which is \texorpdfstring{$\theta$}{theta}-dg enhanced}
\author{Chencheng Zhang}
\address{School of Mathematical Sciences, Shanghai Jiao Tong University}
\email{zhangchencheng@sjtu.edu.cn}
\thanks{The author is supported by the National Natural Science Foundation of China (No.~12131015).}
\date{}

\begin{document}
\raggedbottom
\begin{abstract}
For every nonzero commutative ring $R$ and infinite cardinal $\theta$, we construct a nontopological triangulated category over $R$, whose thick subcategories generated by $<\theta$ objects admit $R$-linear dg enhancements, while the scalar action of $R$ is faithful.
\end{abstract}
\maketitle
\section{Introduction and main result}

A triangulated category is topological if it admits a fully faithful triangle functor into the homotopy category of a stable model category.
The first nontopological example was found by Muro, Schwede and Strickland, who constructed a non-topological triangulation on finitely generated free $\mathbb Z/4$-modules, with identity suspension \cite[Theorem~1]{MSS}.
Their Hopf-map obstruction shows that every triangle functor from this category to a topological triangulated category is zero \cite[Theorem~2]{MSS}.
Rizzardo and Van den Bergh constructed nonenhanceable triangulated categories by gluing algebraic ones along a triangle functor that does not lift to an $A_\infty$-functor \cite[Theorem~1.1]{RVdB}.
Recently D\'iaz Cabrera and Muro constructed a non-topological pretriangulation on finitely generated projective modules over the preprojective algebra $P(\mathbb L_2)$ by twisting the inverse suspension of an algebraic triangulation \cite{DM2026}.
He and Le proved that this pretriangulation satisfies the octahedral axiom over every algebraically closed field \cite[Theorem~5.6]{HL2026}.

For a infinite cardinal $\theta$, we call a triangulated category \emph{$\theta$-dg enhanced} if every thick subcategory generated by a subset $S$ with $|S| < \theta$ is algebraic.
We ask whether triangulated categories with $\theta$-dg enhancement are algebraic or topological.

The following theorem gives a negative answer.

\begin{maintheorem}\phantomsection\label{thm_main}\label{thm_incidence_local_global}
Let $R$ be a nonzero commutative unital ring and let $\theta$ be an infinite cardinal. 
There is an triangulated category $\mathcal T$ such that:
\begin{enumerate}
\item $\mathcal T$ is not topological;
\item $\mathcal T$ is $\theta$-dg enhanced, i.e., for each subset $S$ with $|S| < \theta$, $\operatorname{thick}_{\mathcal T}(S)$ is algebraic;
\item The scalar action $R \to Z(\mathcal T)$ is faithful.
\end{enumerate}
\end{maintheorem}

\begin{corollary}
For any field $k$, there is a nontopological triangulated category whose finitely generated thick subcategories are algebraic.
\end{corollary}

\begin{remark}
    For every cardinal $\kappa$, neither the class of algebraic triangulated categories nor the class of topological triangulated categories can be characterised by diagram-extension axioms, each involving at most $\kappa$ objects and morphisms.
\end{remark}

By contrast, a pretriangulated category is triangulated if every finitely generated thick subcategory is triangulated with respect to the inherited structure.
Roughly speaking, there are no easy definition for algebraic to topological triangulated categories.

\medskip

We briefly describe the construction, and the obstruction to topological realisation.

First, we define $\mathcal C$ to be the $R$-linearisation of $\omega_3$, the first ordinal $\omega_3$ with cardinality $\aleph_3$.
Then we take a graded $R$-category $\mathcal B$ by adding degree $(-2)$-terms on $\mathcal{H}om$, and define two minimal $A_\infty$-categories $\Lambda_\pm$ on $\mathcal B$, with distinct $m_4$.
Then one has $\operatorname{Perf}(\Lambda_\pm) \subseteq \mathcal V_\pm \subseteq D_\infty(\Lambda_\pm)$, where $\mathcal V_\pm $ are generated by $(< \theta)$ copies of coproducts of representable functors.

Fix $I = \theta \times \{+,-\},\quad \operatorname{Sign}(\xi,b)=b$.
The final category $\mathcal T$ is a colimit of $\mathcal T_J = \left(\prod_{i \in J} \mathcal V_{\operatorname{Sign}(i)}\right) \times \operatorname{Perf}(\Lambda_-(I \setminus J))$ with $|J| < \aleph_0$, i.e., the union of algebraic triangulated categories with ``only finite large branches''. 
Some set-theoritic analysis shows that $\mathcal T$ is $\theta$-dg enhanced.

The topological obstruction is based several ``bridges'' which connects fourfold Toda brackets, Yoneda extensions, and categorical cohomology of $\omega_3$.
Hence the difference of $m^\pm_4$ admits several representations, yielding a map $(\chi(\mathfrak t_i))_{i\in I}$ from $I$.
An analysis on infinite combinatorics shows that $(\chi(\mathfrak t_i))_{i\in I}$ is constant except on a subset $S \subseteq I$ with $|S|< \theta$.
Meanwhile standard computation on triangulated categories shows that $(\chi(\mathfrak t_i))_{i\in I}$ vanishes on $\theta \times \{-\}$, while is nonvanishing on $\theta \times \{+\}$.
This yields the contradiction.

Section~\ref{sec_preliminaries} provides preliminaries;
Section~\ref{sec_field_branches} begins with an outline followed by the proofs.

\section{Preliminaries}\label{sec_preliminaries}

Fix a commutative ring $R$.
Unless otherwise specified, $\otimes$ denotes $\otimes_R$.
We use the Koszul sign rule $(f\otimes g)(x\otimes y)=(-1)^{|g||x|}f(x)\otimes g(y)$ for homogeneous maps and elements.
All $A_\infty$-categories, functors, modules and module morphisms are strictly unital.
The triangulated categories constructed in the main theorem are idempotent complete.

An \emph{$R$-category} is a category whose Hom sets are $R$-modules, and the composition of morphisms is $R$-bilinear.
A functor between $R$-linear categories is \emph{$R$-linear}, if it induces $R$-homomorphisms on Hom groups.

\subsection{Triangulated categories}

We use Verdier's axioms and the usual suspension-compatible triangle functors \cite{Verdier1996}.
Our cone convention is $\operatorname{Cone}(f)^i=N^i\oplus M^{i+1}$ with $d(y,x)=(d_Ny+f(x),-d_Mx)$ for a chain map $f:M\to N$.
Write $Z(\mathcal T)$ for the natural endomorphisms of $\mathrm{id}_{\mathcal T}$ commuting with suspension; an $R$-linear scalar action is faithful when $r\mapsto(r1_X)_X$ defines an injection $R\to Z(\mathcal T)$.
A full triangulated subcategory is \emph{thick} if it is closed under direct summands, and $\operatorname{thick}(\mathcal X)$ denotes the thick subcategory generated by $\mathcal X$.
A thick subcategory $\mathcal T_0$ is \emph{finitely generated} if $\mathcal T_0=\operatorname{thick}(G)$ for some object $G$, called a \emph{classical generator} \cite[\S2.1]{BondalVanDenBergh}.

We use \emph{exact categories} in Quillen's sense \cite[Definition~2.1]{Buhler}.
A \emph{Frobenius exact category} has enough projectives and injectives, with the two classes coinciding; its \emph{stable category}, obtained by quotienting out morphisms factoring through these objects, is triangulated \cite[Chapter~I, Theorem~2.6]{Happel}.

\begin{definition}\label{def_algebraic_category}
    A triangulated category is \emph{algebraic (over $R$)} if it is $R$-linearly triangle equivalent to the stable category of an $R$-linear Frobenius exact category.
\end{definition}

A \emph{model category} is understood in the usual complete and cocomplete sense \cite[Definitions~1.1.3--1.1.4]{HoveyModelCategories}, originating in \cite{Quillen}.
Its \emph{homotopy category} is $\operatorname{Ho}(\mathcal M)=\mathcal M[\mathrm{Weq}^{-1}]$, where $\mathrm{Weq}$ is the class of weak equivalences.
A pointed model category is \emph{stable} if its derived suspension is an equivalence; its homotopy category is then canonically triangulated \cite[Definition~7.1.1 and Proposition~7.1.6]{HoveyModelCategories}.

\begin{definition}\label{def_topological_category}
    A triangulated category is \emph{topological} if it admits a fully faithful triangle functor into $\operatorname{Ho}(\mathcal M)$ for some stable model category $\mathcal M$ \cite{SchwedeComparison}.
\end{definition}

\subsection{\texorpdfstring{$A_\infty$}{A-infinity}-categories and dg categories}
\begin{definition}\label{def_ainfinity_category}
    Following \cite[\S3.1]{KellerAinf}, an \emph{$A_\infty$-category} $\mathcal A$ has $\mathbb Z$-graded $R$-modules\linebreak
    $\mathcal{H}om_{\mathcal A}(X,Y)$ as its Hom sets, with $R$-multi\-linear operations
    \[
        m_n:\mathcal{H}om_{\mathcal A}(X_1,X_0)\otimes\cdots\otimes
        \mathcal{H}om_{\mathcal A}(X_n,X_{n-1})
        \longrightarrow\mathcal{H}om_{\mathcal A}(X_n,X_0)
        \qquad(n\geq1)
    \]
    of degree $2-n$, satisfying Stasheff identities:
    \[
        \sum_{\substack{r+s+t=n, \\ r,t \geq 0,\  s \geq 1}}(-1)^{r+st}
        m_{r+1+t}(\mathrm{id}^{\otimes r}\otimes m_s
        \otimes\mathrm{id}^{\otimes t})=0.
        \qquad(n\geq1).
    \]
    Each object $X$ is equipped with $1_X\in\mathcal{H}om_{\mathcal A}^0(X,X)$, satisfying
    \[
        m_2(1_Y,a)=a=m_2(a,1_X)
        \qquad\bigl(a\in\mathcal{H}om_{\mathcal A}(X,Y)\bigr).
    \]
    For every composable sequence $(a_1,\ldots,a_n)$, one also has
    \[
        m_n(a_1,\ldots,a_n)=0
        \quad\text{if }n\ne2\text{ and some }a_j\text{ is an identity}.
    \]
    In particular, $m_1(1_X)=0$.
\end{definition}

\begin{definition}\label{def_ainfinity_functor}
    An \emph{$A_\infty$-functor} $F:\mathcal A\to\mathcal B$ between $A_\infty$-categories assigns an object $FX\in\mathcal B$ to each $X\in\mathcal A$ and, for each $n\geq1$ and $X_0,\ldots,X_n\in\mathcal A$, an $R$-linear map
    \[
        F_n:\mathcal{H}om_{\mathcal A}(X_1,X_0)\otimes\cdots\otimes
        \mathcal{H}om_{\mathcal A}(X_n,X_{n-1})
        \longrightarrow\mathcal{H}om_{\mathcal B}(FX_n,FX_0)
    \]
    of degree $1-n$, satisfying, for every $n\geq1$,
    \[
    \begin{aligned}
    &\sum_{\substack{r+s+t=n\\r,t\geq0,\ s\geq1}}
    (-1)^{r+st}F_{r+1+t}
    (\mathrm{id}^{\otimes r}\otimes m_s^{\mathcal A}\otimes\mathrm{id}^{\otimes t})\\
    &\qquad=\sum_{\substack{k\geq1,\ i_1+\cdots+i_k=n\\i_1,\ldots,i_k\geq1}}
    (-1)^{\sum_{j=1}^k(k-j)(i_j-1)}
    m_k^{\mathcal B}(F_{i_1}\otimes\cdots\otimes F_{i_k}).
    \end{aligned}
    \]
    Graded tensor products of maps use the Koszul sign rule; see \cite[\S\S3.4,7.3]{KellerAinf}.
    We also require
    \[
        F_1(1_X)=1_{FX},
        \qquad
        F_n(a_1,\ldots,a_n)=0
        \quad\text{if }n>1\text{ and some }a_j\text{ is an identity}.
    \]
\end{definition}

An \emph{$A_\infty$-algebra} is an $A_\infty$-category with one object.
A \emph{morphism of $A_\infty$-algebras} is an $A_\infty$-functor between the corresponding one-object categories.
Such a morphism $f:A\to B$ is a \emph{quasi-isomorphism} if $f_1:(A,m_1)\to(B,m_1)$ induces an isomorphism on cohomology.

An $A_\infty$-category is \emph{minimal} if $m_1=0$ on every morphism space.

\begin{lemma}\label{lem_minimal_cohomology}
For a minimal $A_\infty$-category $\mathcal A$, objects $X,Y\in\mathcal A$ and $d\in\mathbb Z$,
\[
H^d\!\left(\mathcal{H}om_{\mathcal A}(X,Y),m_1\right)
=\mathcal{H}om_{\mathcal A}(X,Y)^d.
\]
\end{lemma}
\begin{proof}
This follows from $m_1=0$.
\end{proof}

\begin{definition}\label{def_ainfinity_module}
    A \emph{right $A_\infty$-module} over an $A_\infty$-category $\mathcal A$ assigns a graded $R$-module $M(X)$ to every object $X$, with $R$-linear maps
    \[
    m_n^M:M(X_0)\otimes\mathcal{H}om_{\mathcal A}(X_1,X_0)
       \otimes\cdots\otimes\mathcal{H}om_{\mathcal A}(X_{n-1},X_{n-2})
       \longrightarrow M(X_{n-1})
    \]
    of degree $2-n$ for $n\geq1$, satisfying the module identities of \cite[\S4.2]{KellerAinf} on composable strings.
    We require $m_2^M(x,1_X)=x$, and $m_n^M$ vanishes on inputs containing an identity when $n\ne2$.
    The representable module at $Y$ is $M(X)=\mathcal{H}om_{\mathcal A}(X,Y)$ with the category operations.
\end{definition}

\begin{definition}\label{def_ainfinity_module_morphism}
    A \emph{morphism of right $A_\infty$-modules} $f:L\to M$ has $R$-linear components
    \[
    f_n:L(X_0)\otimes\mathcal{H}om_{\mathcal A}(X_1,X_0)
       \otimes\cdots\otimes\mathcal{H}om_{\mathcal A}(X_{n-1},X_{n-2})
       \longrightarrow M(X_{n-1})
    \]
    of degree $1-n$ for $n\geq1$, with the identities and composition of \cite[\S4.2]{KellerAinf} on composable strings.
    For $n>1$, $f_n$ vanishes on inputs containing an identity.
    It is a \emph{quasi-isomorphism} if $f_1$ induces an isomorphism on cohomology at every object.
\end{definition}

\begin{definition}\label{def_dg_category}\label{def_dg_functor}
    A \emph{differential graded (dg) category} is an $A_\infty$-category with $m_n=0$ for $n\geq3$.
    See \cite[\S2.2]{KellerDG}.
    A \emph{dg functor} between dg categories is an $A_\infty$-functor with $F_n=0$ for $n\geq2$.
    A \emph{dg algebra} is a dg category with one object, and a \emph{morphism of dg algebras} is a dg functor between the corresponding one-object categories.
\end{definition}

Write $C_{\mathrm{dg}}(R)$ for the usual dg category of cochain complexes \cite[\S2.2]{KellerDG}.
\begin{definition}\label{def_dg_module}
    A \emph{right dg module} over $\mathcal A$ is a dg functor $\mathcal A^{\mathrm{op}}\to C_{\mathrm{dg}}(R)$.
    The \emph{representable module} associated to $X$ is denoted by $\mathcal{H}om_{\mathcal A}(-,X)$.
\end{definition}

\subsection{Homotopy categories and enhancements}

For an $A_\infty$-algebra $A$, its \emph{graded cohomology module} is
\[
    H^*(A,m_1)=\coprod_{p\in\mathbb Z}H^p(A,m_1).
\]
The product $[a][b]=[m_2(a,b)]$ makes $H^*(A,m_1)$ a graded associative $R$-algebra with unit $[1_A]$; we also denote the induced product by $m_2$.

The \emph{homotopy category} $H^0(\mathcal A)$ has the objects of $\mathcal A$, morphism groups
\[
    \operatorname{Hom}_{H^0(\mathcal A)}(X,Y)=H^0(\mathcal{H}om_{\mathcal A}(X,Y),m_1),
\]
and composition induced by $m_2$.

\begin{definition}\label{def_pretriangulated_category}
    A dg category is \emph{pretriangulated} if the zero module, all integer shifts of representable modules, and cones of closed degree-zero maps between representable modules are homotopy equivalent to representable modules \cite[\S2.4]{Drinfeld}.
\end{definition}

\begin{definition}\label{def_dg_enhancement}
    An \emph{$R$-linear dg enhancement} of a triangulated category $\mathcal T$ is a pretriangulated dg category $\mathcal A$ with an $R$-linear triangle equivalence $H^0(\mathcal A)\simeq\mathcal T$.
\end{definition}

Every small algebraic triangulated category admits a dg enhancement \cite[Theorem~3.8]{KellerDG}; conversely, $H^0(\mathcal A)$ is algebraic for every pretriangulated $R$-linear dg category $\mathcal A$ \cite[\S4.5]{KellerDG}.

\subsection{Hochschild cohomology}

Let $P$ be a $\mathbb Z$-graded $R$-algebra and $N$ a graded $P$-bimodule whose left and right $R$-actions agree.
For graded $R$-modules $M,N$, write
\[
    \mathcal{H}om_R(M,N)^q=\prod_{d\in\mathbb Z}\operatorname{Hom}_R(M^d,N^{d+q}).
\]
For each $p\geq0$ and $q\in\mathbb Z$, write $C_R^{p,q}(P,N)=\mathcal{H}om_R(P^{\otimes p},N)^q$, where $P^{\otimes 0}=R$.
Define the \emph{Hochschild differential} by
\[
    \delta:C_R^{p,q}(P,N)\longrightarrow C_R^{p+1,q}(P,N),
    \qquad f\longmapsto\delta f,
\]
where, for homogeneous $a_1,\ldots,a_{p+1}\in P$,
\[
    \begin{split}
        &(\delta f)(a_1,\ldots,a_{p+1})
        ={}  (-1)^{q|a_1|}a_1f(a_2,\ldots,a_{p+1}) \\
            & \quad +\sum_{i=1}^{p}(-1)^i
        f(a_1,\ldots,a_{i-1},a_i a_{i+1}, a_{i+2},\ldots,a_{p+1})     +(-1)^{p+1}f(a_1,\ldots,a_p)a_{p+1}.
    \end{split}
\]
For $p=0$, we identify $f\in C_R^{0,q}(P,N)$ with $f(1)\in N^q$, so that $(\delta f)(a)=(-1)^{q|a|}af-fa$.
Then $\delta^2=0$, and the \emph{Hochschild cohomology} is $\operatorname{HH}_R^{p,q}(P,N)
    =H^p\bigl(C_R^{\bullet,q}(P,N),\delta\bigr)$; \cite[\S2]{Muro}.
    
For algebras and bimodules concentrated in degree zero, write $\operatorname{HH}_R^p(P,N)=\operatorname{HH}_R^{p,0}(P,N)$.
For an $R$-linear category $\mathcal C$ and a bimodule $E$, the same definitions use products over composable object strings, with $R$-multilinear cochains valued in $E$ and the same end actions and signs. We write $C_R^p(\mathcal C,E)$ and $\operatorname{HH}_R^p(\mathcal C,E)$ when $\mathcal C$ and $E$ are concentrated in degree zero. Normalized cochains vanish on identity inputs.

\subsection{Twisted complexes}

Finite twisted complexes to compute homotopy categories in Proposition~\ref{prop_marked_core_descent}.
Let $\mathcal A$ be an $A_\infty$-category, with the conventions of \cite[\S\S3.6,7.6]{KellerAinf}.

The objects of $\Sigma\mathcal A$ are formal finite coproducts $E=\bigoplus_{i=1}^N X_i[a_i]$, with $X_i\in\mathcal A$ and $a_i\in\mathbb Z$.
Write $R[a]$ for the rank-one graded free $R$-module concentrated in degree $-a$.
Set
\[
    \mathcal{H}om_{\Sigma\mathcal A}(X[a],Y[b])
    =\mathcal{H}om_R(R[a],R[b])\otimes\mathcal{H}om_{\mathcal A}(X,Y)
    \cong\mathcal{H}om_{\mathcal A}(X,Y)[b-a].
\]
For homogeneous $u_i:R[a_i]\to R[a_{i-1}]$ and $f_i\in\mathcal{H}om_{\mathcal A}(X_i,X_{i-1})$, define
\[
    m_n^{\Sigma\mathcal A}(u_1\otimes f_1,\ldots,u_n\otimes f_n)
    =(-1)^{\sum_{i<j}|f_i||u_j|+(2-n)\sum_i|u_i|}\cdot(u_1\cdots u_n)\otimes m_n^{\mathcal A}(f_1,\ldots,f_n),
\]
where $|u_i|=a_i-a_{i-1}$.
Morphisms between finite coproducts are matrices, and the operations extend by summing over composable matrix entries.

For a $\mathbb Z$-graded $R$-module $V$, write $s:V\to V[1]$ for the canonical map of degree $-1$, where $V[1]^d=V^{d+1}$.
For $n\geq1$ and $E_0,\ldots,E_n\in\Sigma\mathcal A$, define the degree-one map $b_n$ by the commutative diagram
\[
    \xymatrix@C=3em@R=1em{
        \displaystyle\bigotimes_{i=1}^n\mathcal{H}om_{\Sigma\mathcal A}(E_i,E_{i-1})
        \ar[r]^-{m_n^{\Sigma\mathcal A}} \ar[d]_-{s^{\otimes n}}
        & \mathcal{H}om_{\Sigma\mathcal A}(E_n,E_0) \ar[d]^-s \\
        \displaystyle\bigotimes_{i=1}^n\bigl(\mathcal{H}om_{\Sigma\mathcal A}(E_i,E_{i-1})[1]\bigr)
        \ar[r]^-{b_n}
        & \mathcal{H}om_{\Sigma\mathcal A}(E_n,E_0)[1].
    }
\]

\begin{definition}\label{def_twisted_complex}
    A \emph{finite twisted complex} over $\mathcal A$ is a pair $(E,q)$, where $E=\bigoplus_{i=1}^N X_i[a_i]\in\Sigma\mathcal A$ and $q=(q_{ij})_{1\leq i,j\leq N}\in\bigl(\mathcal{H}om_{\Sigma\mathcal A}(E,E)[1]\bigr)^0$ is a matrix with
    \[
        q_{ij}\in\bigl(\mathcal{H}om_{\Sigma\mathcal A}(X_j[a_j],X_i[a_i])[1]\bigr)^0,
        \qquad q_{ij}=0\quad(i\geq j).
    \]
    We require $q$ to satisfy the \emph{Maurer--Cartan equation} $\sum_{n\geq1}b_n(q^{\otimes n})=0$ \cite[\S7.6 and \S8.6]{KellerAinf}.
    Since $q_{ij}=0$ for $i\geq j$, one has $b_n(q^{\otimes n})=0$ for $n\geq N$, so the sum is finite.
\end{definition}

The $A_\infty$-category $\operatorname{Tw}(\mathcal A)$ has finite twisted complexes as objects and graded morphism modules $\mathcal{H}om_{\operatorname{Tw}(\mathcal A)}((E,q),(F,r)) =\mathcal{H}om_{\Sigma\mathcal A}(E,F)$.
For $n\geq1$, objects $(E_i,q_i)$ and $x_i\in\mathcal{H}om_{\Sigma\mathcal A}(E_i,E_{i-1})[1]$. Suspended operations are obtained by inserting twists in all positions:
\[
    b_n^{\operatorname{Tw}}(x_1,\ldots,x_n)
    =\sum_{j_0,\ldots,j_n\geq0}
    b_{n+j_0+\cdots+j_n}(q_0^{\otimes j_0},x_1,q_1^{\otimes j_1},\ldots,x_n,q_n^{\otimes j_n}).
\]
A zeroth tensor power inserts no twist, and the triangularity condition makes each sum finite.
The higher compositions are determined by the commutative diagram
\[
    \xymatrix@C=3em@R=1em{
        \displaystyle\bigotimes_{i=1}^n\mathcal{H}om_{\Sigma\mathcal A}(E_i,E_{i-1})
        \ar[r]^-{m_n^{\operatorname{Tw}}} \ar[d]_-{s^{\otimes n}}
        & \mathcal{H}om_{\Sigma\mathcal A}(E_n,E_0) \ar[d]^-s \\
        \displaystyle\bigotimes_{i=1}^n\bigl(\mathcal{H}om_{\Sigma\mathcal A}(E_i,E_{i-1})[1]\bigr)
        \ar[r]^-{b_n^{\operatorname{Tw}}}
        & \mathcal{H}om_{\Sigma\mathcal A}(E_n,E_0)[1].
    }
\]
The Maurer--Cartan equations ensure the $A_\infty$-identities \cite[\S8.6]{KellerAinf}.
The units are diagonal identity matrices, and $X\mapsto(X,0)$ gives a fully faithful $A_\infty$-embedding $\mathcal A\hookrightarrow\operatorname{Tw}(\mathcal A)$ whose components of order $n\geq2$ vanish.

If $R$ is a field, then $H^0(\operatorname{Tw}(\mathcal A))$ is triangulated, with suspension and distinguished triangles induced by shifts and cones of twisted complexes \cite[\S\S7.5--7.6]{KellerAinf}.

For the categories used below, the same construction embeds in the dg category of strictly unital modules over $R$; shifts and cones give the required triangulation. The arbitrary-ring homotopy Yoneda comparison is used explicitly in Lemma~\ref{lem_large_matrix_realization}. We write $\operatorname{Tw}(\mathcal A)=\mathrm{tw}(\mathcal A)$ and $\operatorname{Perf}(\mathcal A)$ for the idempotent completion of $H^0\operatorname{Tw}(\mathcal A)$.

\begin{proposition}[{\cite[\S11.20, \S12.1 and Proposition~13.19]{BLMPretriangulated}}]
\label{prop_twisted_functor_extension}
Every strictly unital $A_\infty$-functor
$f:\mathcal A\to\mathcal B$ extends to an $A_\infty$-functor
\[
\operatorname{Tw}(f):\operatorname{Tw}(\mathcal A)
\longrightarrow\operatorname{Tw}(\mathcal B),
\]
which induces a triangle functor on $H^0$.
\end{proposition}

\subsection{Fourfold Toda brackets}

Let $\mathcal T$ be a triangulated category.
We use the fourfold Toda bracket to distinguish the two branches in Proposition~\ref{prop_branch_toda_cosets}.

\begin{definition}[{\cite[Definitions~4.2 and~4.4]{Sagave}}]\label{def_fourfold_bracket}
    Consider a five-term complex in $\mathcal T$,
    \[
        X_4\xlongrightarrow{d_4}X_3\xlongrightarrow{d_3}X_2
        \xlongrightarrow{d_2}X_1\xlongrightarrow{d_1}X_0.
    \]
    The \emph{fourfold Toda bracket} $\langle d_1,d_2,d_3,d_4\rangle$ consists of the morphisms $\gamma:X_4[2]\to X_0$ for which there exist objects $C,Y$ and maps satisfying the following conditions:
    \begin{enumerate}
        \item The morphism $\gamma$ factors as $X_4[2]\xlongrightarrow{\alpha}Y\xlongrightarrow{\beta}X_0$ for some $\gamma=\beta\circ\alpha$
        \item The following triangles are distinguished:
              \[
                  X_1\xlongrightarrow{u}C\xlongrightarrow{q}X_2[1]\xlongrightarrow{d_2[1]}X_1[1],\qquad
                  C\xlongrightarrow{v}Y\xlongrightarrow{p}X_3[2]\xlongrightarrow{\delta}C[1].
              \]
        \item The following three diagrams commute:
              \[
                  \xymatrix@C=3em@R=1.5em{
                      X_3[2] \ar[r]^-\delta \ar[dr]_-{d_3[2]}
                      & C[1] \ar[d]^-{q[1]}
                      & X_4[2] \ar[r]^-\alpha \ar[dr]_-{d_4[2]}
                      & Y \ar[d]^-p
                      & X_1 \ar[r]^-{v\circ u} \ar[dr]_-{d_1}
                      & Y \ar[d]^-\beta \\
                      & X_2[2] & & X_3[2] & & X_0
                  }
              \]
    \end{enumerate}
    We regard $\langle d_1,d_2,d_3,d_4\rangle$ as a subset of $\operatorname{Hom}_{\mathcal T}(X_4,X_0[-2])$.
\end{definition}

    A full additive subcategory $\mathcal U\subseteq\mathcal T$ is \emph{$2$-split} if $\operatorname{Hom}_{\mathcal T}(P,Q[r])=0$ for all $P,Q\in\mathcal U$ and every odd integer $r$.
    The subcategory need not be closed under suspension.

\begin{proposition}[{\cite[Proposition~4.10]{Sagave}}]\label{prop_fourfold_indeterminacy}
    If $X_0,\ldots,X_4$ in Definition~\ref{def_fourfold_bracket} belong to a $2$-split subcategory, their fourfold Toda bracket is a nonempty coset of
    \[
        \operatorname{im}\operatorname{Hom}_{\mathcal T}(X_4,d_1[-2])
        +\operatorname{im}\operatorname{Hom}_{\mathcal T}(d_4,X_0[-2])
        \subseteq\operatorname{Hom}_{\mathcal T}(X_4,X_0[-2]).
    \]
\end{proposition}

\begin{remark}\label{rem_toda_functoriality}
    Via its suspension isomorphism, a triangle functor $F:\mathcal T\to\mathcal T'$ satisfies
    \[
        F\bigl(\langle d_1,d_2,d_3,d_4\rangle\bigr)
        \subseteq\langle Fd_1,Fd_2,Fd_3,Fd_4\rangle.
    \]
    Equality holds if $F$ is fully faithful and $X_0,\ldots,X_4$ belong to a $2$-split subcategory.
\end{remark}

\subsection{Category cohomology and characteristic classes}\label{sec_maclane_eventuality}

We introduce the category cohomology and characteristic maps used to prove show topological obstructions.

\begin{definition}[{\cite[\S2.1]{Sagave}}]\label{def_category_cohomology}
    Let $\mathcal P$ be a small category and $\mathbb E:\mathcal P^{\mathrm{op}}\times\mathcal P\to\mathrm{Ab}$ a bifunctor.
    The cochain complex $C^\bullet(\mathcal P;\mathbb E)$ has terms
    \[
            C^0(\mathcal P;\mathbb E)=\prod_{P\in\operatorname{Ob}(\mathcal P)}\mathbb E(P,P);\qquad
            C^n(\mathcal P;\mathbb E)=\prod_{P_n\xrightarrow{f_n}\cdots\xrightarrow{f_1}P_0}\mathbb E(P_n,P_0),\quad (n\geq1),
    \]
    where the second product ranges over all strings of $n$ composable morphisms in $\mathcal P$.
    For $c\in C^n(\mathcal P;\mathbb E)$ with $n\geq1$, define $\delta c\in C^{n+1}(\mathcal P;\mathbb E)$ by
    \[
        \begin{split}
            &(\delta c)(f_1,\ldots,f_{n+1})
            ={}  \mathbb E(1,f_1)\bigl(c(f_2,\ldots,f_{n+1})\bigr) \\
                 &\qquad +\sum_{j=1}^n(-1)^j
            c(f_1,\ldots,f_j\circ f_{j+1},\ldots,f_{n+1})  +(-1)^{n+1}\mathbb E(f_{n+1},1)\bigl(c(f_1,\ldots,f_n)\bigr).
        \end{split}
    \]
    In degree zero, $(\delta c)(f:P_1\to P_0)=\mathbb E(1,f)c(P_1)-\mathbb E(f,1)c(P_0)$.

    The \emph{cohomology groups of $\mathcal P$ with coefficients in $\mathbb E$} are $H^n(\mathcal P;\mathbb E)
=H^n(C^\bullet(\mathcal P;\mathbb E),\delta)$.
\end{definition}

No additivity condition is imposed on the cochains.
A cochain on a category with a zero object is \emph{zero-normalised}, if it vanishes whenever an input morphism is zero.

We also write $H_{\mathrm{cat}}^n(\mathcal P,\mathbb E)=H^n(\mathcal P;\mathbb E)$. The characteristic map for the ringoid used below is constructed in Definition~\ref{def_characteristic_map}.

All size statements are relative to fixed universes; essentially small categories are replaced by small equivalent categories when necessary. We use $\oplus$ for finite coproducts and $\coprod$ for arbitrary coproducts. Modules over ordinary ringoids are right modules.

\section{The main construction}\label{sec_field_branches}

\subsection{Outline}

Let $R$ be a nonzero commutative ring.
Fix an infinite cardinal $\theta$ with $\operatorname{cf}(\theta) > \aleph_3$.
We write $\omega_3$ for the initial ordinal of cardinality $\aleph_3$, so $\theta>\aleph_3=|\omega_3|$.

\paragraph{The construction of $\mathcal T_J$'s.}
Let $\mathcal C$ be the $R$-linearisation of the poset category $\omega_3$, i.e., the objects are unchanged but Hom sets are $R$-enriched.
Let $\underline R : \mathcal C^{\mathrm{op}} \to \mathrm{Mod}_R$ be the constant presheaf taking the value $R$. Subsection~\ref{subsec_ordinal_class} takes the $4$-th syzygy $A$ of $\underline R$
\[
\xymatrix@C=1.5em@R=.5em{
\cdots \ar[r]
& P_4 \ar[rr] \ar@{->>}[dr]_-c
&& P_3 \ar[r]
& P_2 \ar[r]
& P_1 \ar[r]
& P_0 \ar@{->>}[r]
& \underline R
\\
&& A \ \ar@{>->}[ur]
}.
\]
Define the bimodule $E: \mathcal C^{\mathrm{op}} \otimes _R \mathcal C \to \mathrm{Mod}_R$ via $E(\alpha,\beta)=A(\alpha)$.
Then $c:P_4\twoheadrightarrow A$ induces $0 \neq [c]\in\operatorname{Ext}_{\mathcal C}^4(\underline R,A)$ by Lemma~\ref{lem_ordinal_nonvanishing}.
Consider the graded category $\mathcal B$ with
\[
\operatorname{Ob}(\mathcal B) = \operatorname{Ob}(\mathcal C) = \omega_3, \qquad  \mathcal{H}om_{\mathcal B}(\alpha,\beta)^d=
\begin{cases}
\operatorname{Hom}_{\mathcal C}(\alpha,\beta),&d=0,\\
E(\alpha,\beta)=A(\alpha),&d=-2,\\
0,&\text{otherwise}.
\end{cases}
\]
We identify $\mathcal{H}om_{\mathcal B}(\alpha,\beta) \cong \operatorname{Hom}_{\mathcal C}(\alpha,\beta) \oplus E(\alpha,\beta)$, and the composition becomes
\[
(g,y)\circ(f,x)
=\bigl(g\circ f,\; g\cdot x + y \cdot f\bigr).
\]
Here, for $f:\alpha\to\beta$ and $g:\beta\to\gamma$, by definition one has
\[
g\cdot x
=\begin{cases}
 s x,&\beta\leq\gamma,\quad (g=s u_{\beta,\gamma}),\\
 0,&\beta>\gamma,
\end{cases}; \qquad
y \cdot f
=\begin{cases}
 r A(u_{\alpha,\beta})(y),
   &\alpha\leq\beta,\quad (f=r u_{\alpha,\beta}),\\
 0,&\alpha>\beta.
\end{cases}
\]
Note that $\mathcal B$ has usual $A_\infty$-structure, with $m_2$ the usual multiplication and $m_{\neq 2}=0$, to obtain $\Lambda_-$;
meanwhile Lemma~\ref{lem_square_zero_detector} gives another $A_\infty$-structure on $\mathcal B$ with $m_2$ the usual multiplication, $m_4=m_4^{+}$, and all other operations zero, to obtain $\Lambda_+$.
For either $b\in\{+,-\}$, let $D_\infty(\Lambda_b)$ denote the derived category by localising its right $A_\infty$-modules at quasi-isomorphisms.
Put $\mathcal V_b=\operatorname{thick}_{D_\infty(\Lambda_b)}(\mathfrak G_b)$, where $\mathfrak G_b$ is the coproduct of $|\omega_3|$ copies of each representable functors $\mathcal{H}om_{\Lambda_b}(-,\alpha)$'s.

By Proposition~\ref{prop_marked_core_descent}, $\operatorname{Perf}(\Lambda_-)\simeq\operatorname{Perf}(\Lambda_+)$;
however, the larger categories $\mathcal V_\pm$ detect the difference between the two $m_4$-operations:
the fourfold Toda brackets associated with the same projective resolution determine the distinct classes $0$ and $-[c]$ in $\operatorname{Ext}_{\mathcal C}^4(\underline R,A)$.

Fix $I = \theta \times \{+,-\}$ with projection $\operatorname{Sign} : I \to \{+,-\}$.
For finite subset $J \subseteq I$, take
\[
\mathcal T_J
=\left(\prod_{i\in J}\mathcal V_{\operatorname{Sign}(i)}\right)
\times\operatorname{Perf}(\Lambda_-(I\setminus J)), \qquad |J|<\aleph_0 \qquad (\text{see Equation~\eqref{eq_TJ}}).
\]
We shall specify a filtered system in Subsection~\ref{subsec_function_coefficients}.
Take its colimit $\mathcal T = \varinjlim\limits_{J \subseteq I, \ |J| < \aleph_0} \mathcal T_J$.

\paragraph{Each $<\theta$-generated category is algebraic.}
Let $(X_s)_{s\in S}$ be a family of objects of $\mathcal T$ with $|S|<\theta$. 
Fix $X_s \in \mathcal T_{J_s}$ for a finite subset $J_s\subseteq I$.
Then $J=\bigcup_{s\in S}J_s$.
Since $|J|< \theta$, Proposition~\ref{prop_small_stage_inclusion} gives an $R$-linear dg enhancement.

\paragraph{However $\mathcal T$ is not topological.}
Suppose that $\mathcal T$ admits a topological realisation.
Proposition~\ref{prop_topological_universal_class} gives a Toda class $\mathfrak t=[u]$.
For each $i\in I$, let $\mathfrak t_i$ be the branch class obtained from $\mathfrak t$ as in Lemma~\ref{lem_corner_prism_naturality}, and $\chi$ the characteristic map of Definition~\ref{def_characteristic_map}.

On one hand, 
Lemma~\ref{lem_finite_large_characteristic}
and Proposition~\ref{prop_corrected_toda_endpoint} give
\[
\chi(\mathfrak t_i)=
\begin{cases}
0,&\operatorname{Sign}(i)=-,\\
-[c],&\operatorname{Sign}(i)=+.
\end{cases}
\]

On the other hand, $(\chi(\mathfrak t_i))_{i \in I} \in\prod_{i\in I} \operatorname{Ext}^4_{\mathcal C} (\underline R, A)$ is constant except for a subset $S \subseteq I$ with $|S| < \theta$ (Proposition~\ref{prop_characteristic_almost_constant}).

These lead a contradiction.

\paragraph{Faithful central action by $R$.}
The constant scalar functions act faithfully on every $\operatorname{Perf}(\mathcal B)$. Therefore $R$ acts on $Z(\mathcal T)$ faithfully.

\subsection{Two distinct \texorpdfstring{$A_\infty$}{} structures over \texorpdfstring{$\mathcal B$}{}}\label{sec_two_structure}
\label{subsec_ordinal_class}

\begin{lemma}\label{lem_tail_simplex_exactness}
For each $\beta<\omega_3$, the following sequence is contractible
\[
\cdots\longrightarrow
\coprod_{\beta\leq\alpha_0<\alpha_1<\alpha_2<\omega_3}R
\xlongrightarrow{d_2}
\coprod_{\beta\leq\alpha_0<\alpha_1<\omega_3}R
\xlongrightarrow{d_1}
\coprod_{\beta\leq\alpha_0<\omega_3}R
\xlongrightarrow{d_0}R\longrightarrow0.
\]
Writing $\alpha_0,\ldots,\alpha_p$ as coordinate basis of the corresponding summand, the differentials are
\[
d_p[\alpha_0,\ldots,\alpha_p]
=\sum_{j=0}^{p}(-1)^j
[\alpha_0,\ldots,\widehat{\alpha_j},\ldots,\alpha_p]
\quad(p\geq1),
\qquad d_0[\alpha_0]=1.
\]
\end{lemma}

\begin{proof}
Define $R$-linear maps by $h_{-1}(1)=[\beta]$ and
\[
h_p[\alpha_0,\ldots,\alpha_p]
=\begin{cases}
[\beta,\alpha_0,\ldots,\alpha_p],&\beta<\alpha_0,\\
0,&\beta=\alpha_0.
\end{cases}
\]
Then $d_pd_{p+1}=0$, $d_0h_{-1}=\mathrm{id}_R$ and $d_{p+1}h_p+h_{p-1}d_p=\mathrm{id}$ for $p \geq 0$, proving the lemma.
\end{proof}

Then gives a resolution of $\underline R$, whose pointwise exactness is checked in Lemma~\ref{lem_tail_simplex_exactness}
\[
\cdots\longrightarrow P_4\xrightarrow{d_4}P_3\xrightarrow{d_3}P_2\xrightarrow{d_2}P_1\xrightarrow{d_1}P_0
 \xrightarrow{d_0}\underline R\longrightarrow0,\qquad P_p=\coprod_{\alpha_0<\cdots<\alpha_p<\omega_3}y_{\alpha_0}.
\]
Put $A=\operatorname{im}(d_4:P_4\to P_3)$ and let $c:P_4\to A$ be the corestriction of $d_4$.

\begin{lemma}\label{lem_ordinal_flat_dimension}
The module $\underline R$ is flat and
$\operatorname{pd}_{\mathcal C}\underline R=4$.
\end{lemma}

\begin{proof}
    Since $\underline R = \varinjlim _{\beta < \omega_3}y_\beta$ is a filtered colimit of projective modules, it is flat.

    Since $\aleph_3$ is regular, $\omega_3$ has cofinality $\aleph_3$.
    By Mitchell's theorem \cite[Theorem~A]{Mitchell}, $\omega_3$ has $R$-cohomological dimension $4$, i.e., $\varprojlim ^p N$ vanishes for any inverse system $N : (\omega_3)^{\mathrm{op}} \to \mathrm{Mod}_R$ precisely when $p > 4$.
    Since limits are right adjoint to diagonal functor, $\operatorname{Hom}_{\mathcal C}(\underline R,-) \cong\varprojlim_{\alpha<\omega_3}(-)(\alpha)$.
    This yields exactly $\operatorname{pd}_{\mathcal C}\underline R=4$.
\end{proof}

\begin{lemma}\label{lem_ordinal_nonvanishing}
The cocycle $c: P_4\to A$ represents a nonzero class in $\operatorname{Ext}_{\mathcal C}^4(\underline R,A)$.
\end{lemma}

\begin{proof}
If $[c]=0$, dimension shifting implies that $0\to A\to P_3\to\operatorname{im}d_3\to0$ splits.
Thus $\operatorname{im}d_3$ is projective, contradicting
$\operatorname{pd}_{\mathcal C}\underline R=4$.
\end{proof}

We construct two distinct $A_\infty$-structures over the graded category $\mathcal B$ as defined in the outline.
The first $A_\infty$ cateogry $\Lambda_-$ has underlying graded category $\mathcal B$, $m_2^{-}$ the usual composition and $m_{\neq 2}^{-}=0$.

\begin{definition}
    Let $\Lambda_+$ be an $A_\infty$ category with underlying graded category $\mathcal B$, $m_2^{+}$ the usual multiplication, $m_{i}=0$ for $i \notin \{2,4\}$, and $m_4$ is defined as follows:
    \begin{enumerate}
        \item For $\alpha_0<\alpha_1<\alpha_2<\alpha_3<\alpha_4 < \omega_3$, define 
        \[
m_4^{+}(u_{\alpha_3,\alpha_4},u_{\alpha_2,\alpha_3},u_{\alpha_1,\alpha_2},u_{\alpha_0,\alpha_1})
=c_{\alpha_0}[\alpha_0,\alpha_1,\alpha_2,\alpha_3,\alpha_4],
\]
where the result belongs to $A(\alpha_0) =E(\alpha_0, \alpha_4)=\mathcal{H}om_{\mathcal B}(\alpha_0,\alpha_4)^{-2}$.
\item Set $m_4^{+}(f_1,f_2,f_3,f_4)=0$ if one of the morphism is $1$, or has degree $-2$.
\item Extend $m_4$ to make it linear in $4$ arguments.
    \end{enumerate}
\end{definition}

\begin{lemma}\label{lem_square_zero_detector}
The categories $\Lambda_-$ and $\Lambda_+$ are strictly unital minimal $R$-linear $A_\infty$ categories.
\end{lemma}

\begin{proof}
For $\Lambda_+$, we check Stasheff identities for $\Lambda_+$, where the identity for arity $N$ is
\[
\sum_{r+s+t=N}(-1)^{r+st}
 m_{r+1+t}^{+}(\mathrm{id}^{\otimes r}\otimes m_s^{+}\otimes\mathrm{id}^{\otimes t})=0.
\]
It suffices to check the case $r+1+t, s \in \{2,4\}$, where $N \in \{3,5,7\}$.
\begin{enumerate}
    \item[($N = 3$)] Associativity yields $m_2^{+}\bigl(m_2^{+}(f_3,f_2),f_1\bigr) -m_2^{+}\bigl(f_3,m_2^{+}(f_2,f_1)\bigr)=0$.
    \item[($N=5$)] Let $\mathfrak L$ denote the left hand side of the identity, i.e., 
    \[
\begin{aligned}
\mathfrak L={}&m_2^{+}\bigl(m_4^{+}(f_5,f_4,f_3,f_2),f_1\bigr)
-m_2^{+}\bigl(f_5,m_4^{+}(f_4,f_3,f_2,f_1)\bigr)\\
&+m_4^{+}\bigl(m_2^{+}(f_5,f_4),f_3,f_2,f_1\bigr)-m_4^{+}\bigl(f_5,m_2^{+}(f_4,f_3),f_2,f_1\bigr)\\
&+m_4^{+}\bigl(f_5,f_4,m_2^{+}(f_3,f_2),f_1\bigr)-m_4^{+}\bigl(f_5,f_4,f_3,m_2^{+}(f_2,f_1)\bigr).
\end{aligned}
\]
We first show the case $f_j=u_{\alpha_{j-1}\alpha_j}$ with $\alpha_0 < \cdots < \alpha _5 < \omega_3$:
\[
\begin{aligned}
\mathfrak L&=A(u_{\alpha_0,\alpha_1})\left(c_{\alpha_1}[\alpha_1,\ldots,\alpha_5]\right)
-c_{\alpha_0}[\alpha_0,\ldots,\alpha_4] +c_{\alpha_0}[\alpha_0,\alpha_1,\alpha_2,\alpha_3,\alpha_5]\\
&\quad -c_{\alpha_0}[\alpha_0,\alpha_1,\alpha_2,\alpha_4,\alpha_5]
+c_{\alpha_0}[\alpha_0,\alpha_1,\alpha_3,\alpha_4,\alpha_5]
-c_{\alpha_0}[\alpha_0,\alpha_2,\alpha_3,\alpha_4,\alpha_5]\\
&=c_{\alpha_0}\!\left(
\sum_{j=0}^{5}(-1)^j
[\alpha_0,\ldots,\widehat{\alpha_j},\ldots,\alpha_5]
\right)=(c\circ d_5)_{\alpha_0}[\alpha_0,\ldots,\alpha_5]=0.
\end{aligned}
\]
If $f_i = 1$ or $f_i$ has degree $-2$ for some $f_i$, then is easy to check that $\mathfrak L = 0$.
    \item[($N=7$)] The identity becomes
    \[
\begin{aligned}
0={}&m_4^{+}\bigl(m_4^{+}(f_7,f_6,f_5,f_4),f_3,f_2,f_1\bigr)-m_4^{+}\bigl(f_7,m_4^{+}(f_6,f_5,f_4,f_3),f_2,f_1\bigr)\\
&+m_4^{+}\bigl(f_7,f_6,m_4^{+}(f_5,f_4,f_3,f_2),f_1\bigr)-m_4^{+}\bigl(f_7,f_6,f_5,m_4^{+}(f_4,f_3,f_2,f_1)\bigr).
\end{aligned}
\]
Since the output of $m_4^{+}$ has degree $-2$, the right hand side vanishes.
\end{enumerate}
Note that $m_1=0$ gives minimality.
To see the strict units, it is clear to verify
\[
m_1^{+}(1)=0,\qquad
m_2^{+}(1,f)=f
=m_2^{+}(f,1),\qquad
m_{\geq 3}^{+}(\ldots,1,\ldots)=0.
\]
For $\Lambda_-$ the verification is much easier.
\end{proof}

\subsection{Equivalence of the perfect categories}
\label{subsec_core_descent}

For a subset $\Gamma\subset\omega_3$, let $\mathcal C_{\Gamma}$ be the full subcategory of $\mathcal C$ on $\Gamma$, and let $E_{\Gamma}$ be the restriction of $E$ to $\mathcal C_{\Gamma}^{\mathrm{op}}\otimes_R\mathcal C_{\Gamma}$.

Based on the distinct $A_\infty$-structures $\Lambda_\pm$, we use twisted complexes show the equivalence $\left(H^0\operatorname{Tw}(\Lambda_+|_{\Gamma})\right)_{\Gamma}\cong\left(H^0\operatorname{Tw}(\Lambda_-|_{\Gamma})\right)_{\Gamma}$ for each finite subset $\Gamma \subseteq \omega_3$.
The we show these equivalences are compatible with inclusions of finite subsets, and ``glue'' these equivalences to $H^0\operatorname{Tw}(\Lambda_+) \simeq H^0\operatorname{Tw}(\Lambda_-)$.
Finally idempotent completion gives $\operatorname{Perf}(\Lambda_+)
\simeq\operatorname{Perf}(\Lambda_-)$.

\begin{lemma}\label{lem_finite_ordinal_acyclicity}
For every nonempty finite set $\Gamma\subset\omega_3$,
$\operatorname{HH}_R^p(\mathcal C_{\Gamma},E_{\Gamma})=0$ for $p>0$.
\end{lemma}
\begin{proof}
By definition, $\operatorname{HH}_R^p(\mathcal C_{\Gamma},E_{\Gamma}) =H^p\bigl(C_R^\bullet(\mathcal C_{\Gamma},E_{\Gamma}),\delta\bigr)$, and
\[
\begin{aligned}
C_R^p(\mathcal C_{\Gamma},E_{\Gamma})\cong\prod_{\alpha_0\le\cdots\le\alpha_p\atop\alpha_i\in \Gamma}
\operatorname{Hom}_R\!\left(R^{\otimes_R p},A(\alpha_0)\right)\cong\prod_{\alpha_0\le\cdots\le\alpha_p\atop\alpha_i\in \Gamma}
A(\alpha_0).
\end{aligned}
\]
Take $\beta = \sup \Gamma$. Then $\beta < \omega_3$.
For each $p \geq 1$, put 
\[
h:C_R^p(\mathcal C_{\Gamma},E_{\Gamma})\longrightarrow C_R^{p-1}(\mathcal C_{\Gamma},E_{\Gamma}),\qquad
(hf)_{\alpha_0,\ldots,\alpha_{p-1}}
=f_{\alpha_0,\ldots,\alpha_{p-1},\beta}.
\]
Then by construction one has $\delta h +h\delta =1$, since
\[
\begin{aligned}
(\delta hf)_{\alpha_0,\ldots,\alpha_p}
={}&f_{\alpha_0,\ldots,\alpha_{p-1},\beta}
+\sum_{j=1}^{p-1}(-1)^{p-j}
 f_{\alpha_0,\ldots,\widehat{\alpha_j},\ldots,\alpha_p,\beta}+(-1)^p A(u_{\alpha_0,\alpha_1})
 f_{\alpha_1,\ldots,\alpha_p,\beta},\\
(h\delta f)_{\alpha_0,\ldots,\alpha_p}
={}&f_{\alpha_0,\ldots,\alpha_p}
-f_{\alpha_0,\ldots,\alpha_{p-1},\beta}
+\sum_{j=1}^{p-1}(-1)^{p+1-j}
 f_{\alpha_0,\ldots,\widehat{\alpha_j},\ldots,\alpha_p,\beta}+(-1)^{p+1}A(u_{\alpha_0,\alpha_1})
 f_{\alpha_1,\ldots,\alpha_p,\beta}.
\end{aligned}
\]
By setwise homotopy \cite{BianZhangZhang2009}, $C_R^p(\mathcal C_{\Gamma},E_{\Gamma})$ is acyclic, thus $\operatorname{HH}_R^p(\mathcal C_{\Gamma}, E_{\Gamma})=0$.
\end{proof}

Hence $m_4^+|_{\Gamma} = \delta h_{\Gamma}$ for some $h_{\Gamma}\in C_R^3(\mathcal C_{\Gamma},E_{\Gamma})$.
Fix such $h_{\Gamma}$, and denote $\Lambda_\pm|_{\Gamma}$ as the full subcategories on $\Gamma$.
Then $h_{\Gamma}$ gives an $A_\infty$-functor $f_{\Gamma}: \Lambda_+|_{\Gamma}\to\Lambda_-|_{\Gamma}$ which satisfies
\begin{enumerate}
    \item $f_{\Gamma}$ is identical on objects, and $(f_{\Gamma})_1 = \mathrm{id}$, i.e., $f_{\Gamma}$ is identical on morphisms;
    \item $(f_{\Gamma})_3 = h_{\Gamma}$, and $(f_{\Gamma})_i=0$ for $i \notin\{1,3\}$.
\end{enumerate}
Here we extend $h_{\Gamma}$ to $\Lambda_+ |_{\Gamma}$ by sending the input in $E$ to $0$.

\begin{lemma}
    For any nonempty finite subset $\Gamma \subseteq \omega_3$, $f_{\Gamma} : \Lambda_+|_{\Gamma}\to\Lambda_-|_{\Gamma}$ is an $A_\infty$ functor. 
\end{lemma}

\begin{proof}
    Write $f_{\Gamma} = f$, $h = h_{\Gamma}$ and $m_i^\pm |_{\Gamma} = m_i^\pm$ for convenience in this proof.
    Since $m_i^{-} = 0$ for $i \neq 2$, and $(f_{\Gamma})_i = 0$ for $i \notin \{1,3\}$, it suffices to check the defining equalities with $N$ arguments, where $N \in \{2,4,6\}$.
    \begin{enumerate}
        \item[$(N=2)$] The identity is $f_1 m^+_2 = m^-_2 (f_1 \otimes f_1)$, which follows from $m_2^+ = m_2^-$. 
        \item[$(N=4)$] The identity is $f_1m_4^++f_3(m_2^+\otimes1\otimes1)-f_3(1\otimes m^+_2\otimes1) +f_3(1\otimes1\otimes m^+_2)=m_2^-(f_1\otimes f_3)+m_2^-(f_3\otimes f_1)$.
        Use $f_1=\mathrm{id}$, $f_3=h_{\Gamma}$, $m_4^+=\delta h_{\Gamma}$ and $m_2^- = m_2^+$.
        Then
        \[
\begin{aligned}
&f_1m_4^+
+f_3(m_2^+\otimes1\otimes1)
-f_3(1\otimes m_2^+\otimes1)
+f_3(1\otimes1\otimes m_2^+)\\
={}&\delta h_{\Gamma}
+h_{\Gamma}(m_2\otimes1\otimes1)
-h_{\Gamma}(1\otimes m_2\otimes1)
+h_{\Gamma}(1\otimes1\otimes m_2)\\
={}&m_2(1\otimes h_{\Gamma})
-h_{\Gamma}(m_2\otimes1\otimes1)
+h_{\Gamma}(1\otimes m_2\otimes1)
-h_{\Gamma}(1\otimes1\otimes m_2)\\
&\qquad+m_2(h_{\Gamma}\otimes1)+h_{\Gamma}(m_2\otimes1\otimes1)
-h_{\Gamma}(1\otimes m_2\otimes1)
+h_{\Gamma}(1\otimes1\otimes m_2)\\
={}&m_2(1\otimes h_{\Gamma})+m_2(h_{\Gamma}\otimes1) \qquad =m_2^-(f_1\otimes f_3)+m_2^-(f_3\otimes f_1).
\end{aligned}
\]
\item[$(N=6)$] The identity is $f_3(m_4^+\otimes1\otimes1)
-f_3(1\otimes m_4^+\otimes1)
+f_3(1\otimes1\otimes m_4^+)=m_2^-(f_3\otimes f_3)$.
Since both $m_4^+$ and $f_3$ take values in $E$ and vanishes for the inputs from $E$, the left hand side vanishes; by $E^2=0$ the right hand side also vanishes. \qedhere
    \end{enumerate}
\end{proof}

\begin{lemma}\label{lem_finite_twisted_equivalence}
For every nonempty finite $\Gamma\subset\omega_3$, the $A_\infty$-functor
$f_{\Gamma}:\Lambda_+|_{\Gamma}\to\Lambda_-|_{\Gamma}$ constructed above induces a triangle equivalence
\[
F_{\Gamma}:=H^0\operatorname{Tw}(f_{\Gamma}):
H^0\operatorname{Tw}(\Lambda_+|_{\Gamma})
\xrightarrow{\sim}
H^0\operatorname{Tw}(\Lambda_-|_{\Gamma}),
\]
which fixes the one-term objects and their graded morphism groups, i.e.,
\[
F_{\Gamma}(\alpha,0)=(\alpha,0), \quad 
\operatorname{Hom}_{H^0\operatorname{Tw}(\Lambda_+|_{\Gamma})}
((\alpha,0),(\beta,0)[d])
\xlongrightarrow{\sim}
\operatorname{Hom}_{H^0\operatorname{Tw}(\Lambda_-|_{\Gamma})}
((\alpha,0),(\beta,0)[d]).
\]
\end{lemma}

\begin{proof}
Proposition~\ref{prop_twisted_functor_extension} gives the triangle functor $F_{\Gamma}=H^0\operatorname{Tw}(f_{\Gamma})$.
Since $f_{\Gamma}$ is the identity on objects, $F_{\Gamma}(\alpha,0)=(f_{\Gamma}(\alpha),0)=(\alpha,0)$.
For $\alpha,\beta\in \Gamma$ and $d\in\mathbb Z$,
\[
\operatorname{Hom}_{H^0\operatorname{Tw}(\Lambda_\pm|_{\Gamma})}
\bigl((\alpha,0),(\beta,0)[d]\bigr)
\cong H^d\!\left(\mathcal{H}om_{\Lambda_\pm}(\alpha,\beta),m_1^{\pm}\right)=\mathcal{H}om_{\mathcal B}(\alpha,\beta)^d.
\]
The first isomorphism follows from the definition of twisted complexes, and the equality follows from Lemma~\ref{lem_minimal_cohomology}.
Under these identifications, for the objects $(\alpha,0)$, $(\beta,0)$ and $a\in\mathcal{H}om_{\mathcal B}(\alpha,\beta)^d$, the formula defining $\operatorname{Tw}(f_{\Gamma})$ gives $\bigl(\operatorname{Tw}(f_{\Gamma})\bigr)_1(a)=(f_{\Gamma})_1(a)=a$.
Thus $F_{\Gamma}$ induces the identity after identifying both sides with $\mathcal{H}om_{\mathcal B}(\alpha,\beta)^d$
\[
\operatorname{Hom}_{H^0\operatorname{Tw}(\Lambda_+|_{\Gamma})}
((\alpha,0),(\beta,0)[d])
\longrightarrow
\operatorname{Hom}_{H^0\operatorname{Tw}(\Lambda_-|_{\Gamma})}
((\alpha,0),(\beta,0)[d]),\qquad
[a]\longmapsto[a].
\]

For arbitrary $T,U\in H^0\operatorname{Tw}(\Lambda_+|_{\Gamma})$, consider the natural comparison maps
\[
\operatorname{Hom}_{H^0\operatorname{Tw}(\Lambda_+|_{\Gamma})}(T,U[d])
\longrightarrow
\operatorname{Hom}_{H^0\operatorname{Tw}(\Lambda_-|_{\Gamma})}(F_{\Gamma}(T),F_{\Gamma}(U)[d]).
\]
The isomorphism holds for $T=(\alpha,0)$ and $U = (\beta,0)$, and thus extends to finite twisted complexes by taking (de)suspensions and mapping cones (using five lemma).

Thus $F_{\Gamma}$ is fully faithful. It is essentially surjective, thus a triangle equivalence.
\end{proof}

For every nonempty finite $\Gamma\subset\omega_3$, fix the normalized cochains $h_{\Gamma}$ used to define $f_{\Gamma}$.

\begin{lemma}\label{lem_twisted_overlap_coherence}
For any inclusion of finite subsets $\Gamma \subseteq \Gamma' (\subseteq \omega_3)$, let $\iota^\pm_{\Gamma,\Gamma'}$ denote $H^0\operatorname{Tw}(\Lambda_\pm |_{\Gamma}) \to H^0\operatorname{Tw}(\Lambda_\pm |_{\Gamma'})$.
Then there is a natural isomorphism $\vartheta_{\Gamma,\Gamma'} : \iota^-_{\Gamma,\Gamma'} F_{\Gamma} \Rightarrow F_{\Gamma'} \iota_{\Gamma,\Gamma'}^+$ 
\[
\xymatrix@C=4em@R=2em{
H^0\operatorname{Tw}(\Lambda_+|_{\Gamma})
\ar[r]^{F_{\Gamma}}\ar[d]_{\iota^+_{\Gamma,\Gamma'}}
& H^0\operatorname{Tw}(\Lambda_-|_{\Gamma})
\ar[d]^{\iota^-_{\Gamma,\Gamma'}}
\ar@{}[dl]^(.3){}="theta-start"^(.7){}="theta-end"
\ar@{=>}"theta-start";"theta-end"|{\vartheta_{\Gamma,\Gamma'}}\\
H^0\operatorname{Tw}(\Lambda_+|_{\Gamma'})
\ar[r]_{F_{\Gamma'}}
& H^0\operatorname{Tw}(\Lambda_-|_{\Gamma'})
}
\]
such that $\vartheta_{\Gamma,\Gamma}=1$ and $\vartheta_{\Gamma,\Gamma''}=(\vartheta_{\Gamma',\Gamma''} \iota_{\Gamma,\Gamma'}^+)\circ(\iota_{\Gamma',\Gamma''}^-\vartheta_{\Gamma,\Gamma'})$ for finite subsets $\Gamma\subseteq \Gamma'\subseteq \Gamma'' (\subseteq \omega_3)$.
These isomorphisms commute with the suspensions.
\end{lemma}
\begin{proof}
Since $\delta h_{\Gamma}=m_4^+|_{\Gamma}$, the difference $h_{\Gamma'}|_{\Gamma}-h_{\Gamma}$ is a normalized Hochschild cocycle.
Lemma~\ref{lem_finite_ordinal_acyclicity} gives a normalized $k_{\Gamma,\Gamma'}\in C_R^2(\mathcal C_{\Gamma},E_{\Gamma})$ with $\delta k_{\Gamma,\Gamma'}=h_{\Gamma'}|_{\Gamma}-h_{\Gamma}$.
Choose $k_{\Gamma,\Gamma}=0$.
Extend $k_{\Gamma,\Gamma'}$ by zero on $E$ inputs, and to matrices:
\[
\bigl(k_{\Gamma,\Gamma'}(B,A)\bigr)_{\gamma\alpha}
=\sum_\beta k_{\Gamma,\Gamma'}(B_{\gamma\beta},A_{\beta\alpha}), \qquad (A:X \to Y, \quad B : Y \to Z \text{ are matrices}).
\]
For canonical map $s:V\to V[1]$, write $k_{\Gamma,\Gamma'}^s s^{\otimes2}=sk_{\Gamma,\Gamma'}$.
For $X=(V,q)$, define
\[
t_{\Gamma,\Gamma',X}=1_V-s^{-1}k_{\Gamma,\Gamma'}^s(q,q),\qquad
(\vartheta_{\Gamma,\Gamma'})_X=[t_{\Gamma,\Gamma',X}]
\in\operatorname{Hom}_{H^0\operatorname{Tw}(\Lambda_-|_{\Gamma'})}(F_{\Gamma}X,F_{\Gamma'}X).
\]
Here $|k_{\Gamma,\Gamma'}^s| = -1$, thus $|t_{\Gamma,\Gamma',X}| = 0$.
This defines $\vartheta$.

We then show $t_{\Gamma,\Gamma',X}$ is closed, so $(\vartheta_{\Gamma,\Gamma'})_X = [t_{\Gamma,\Gamma',X}]$ is a morphism in $H^0\operatorname{Tw}(\Lambda_-|_{\Gamma'})$.
Put $q_{\Gamma}=q+h_{\Gamma}^s(q,q,q)$ and $q_{\Gamma'}=q+(h_{\Gamma'}|_{\Gamma})^s(q,q,q)$.
We show
\[
\big(1_V - s^{-1}k_{\Gamma,\Gamma'}^s(q,q)\big) \in Z^0 \mathcal{H}om_{\operatorname{Tw}(\Lambda_-|_{\Gamma'})}
\bigl((V,q_{\Gamma}),(V,q_{\Gamma'})\bigr).
\]
Let $d$ denote the differential on $\mathcal{H}om_{\operatorname{Tw}(\Lambda_-|_{\Gamma'})} ((V,q_{\Gamma}),(V,q_{\Gamma'}))$.
Write $b_2 =s\circ m_2\circ(s\otimes s)^{-1}$.
Using $(\delta k_{\Gamma,\Gamma'})^s s^{\otimes3}=s(\delta k_{\Gamma,\Gamma'})$, and evaluating at $(q,q,q)$ gives
\[
(\delta k_{\Gamma,\Gamma'})^s(q,q,q)
={}-b_2\bigl(q,k_{\Gamma,\Gamma'}^s(q,q)\bigr)
    -b_2\bigl(k_{\Gamma,\Gamma'}^s(q,q),q\bigr)-k_{\Gamma,\Gamma'}^s\bigl(b_2(q,q),q\bigr)
    -k_{\Gamma,\Gamma'}^s\bigl(q,b_2(q,q)\bigr).
\]
The Maurer--Cartan equation for $q$ is
$b_2(q,q)+b_4^+(q,q,q,q)=0$.
Using $E^2=0$, the vanishing of $k_{\Gamma,\Gamma'}$ on $E$ inputs,
and $\delta k_{\Gamma,\Gamma'}=h_{\Gamma'}|_{\Gamma}-h_{\Gamma}$, we obtain
\begin{align*}
s\,d(t_{\Gamma,\Gamma',X})
&=b_2(q_{\Gamma'},st_{\Gamma,\Gamma',X})+b_2(st_{\Gamma,\Gamma',X},q_{\Gamma})\\
&=q_{\Gamma}-q_{\Gamma'}
  -b_2\bigl(q_{\Gamma'},k_{\Gamma,\Gamma'}^s(q,q)\bigr)
  -b_2\bigl(k_{\Gamma,\Gamma'}^s(q,q),q_{\Gamma}\bigr)\\
&=q_{\Gamma}-q_{\Gamma'}
  -b_2\bigl(q,k_{\Gamma,\Gamma'}^s(q,q)\bigr)
  -b_2\bigl(k_{\Gamma,\Gamma'}^s(q,q),q\bigr)\\
&=q_{\Gamma}-q_{\Gamma'}+(\delta k_{\Gamma,\Gamma'})^s(q,q,q)
  +k_{\Gamma,\Gamma'}^s\bigl(b_2(q,q),q\bigr)
  +k_{\Gamma,\Gamma'}^s\bigl(q,b_2(q,q)\bigr)\\
&=q_{\Gamma}-q_{\Gamma'}+(\delta k_{\Gamma,\Gamma'})^s(q,q,q)
  -k_{\Gamma,\Gamma'}^s\bigl(b_4^+(q,q,q,q),q\bigr)
  -k_{\Gamma,\Gamma'}^s\bigl(q,b_4^+(q,q,q,q)\bigr)\\
&=q_{\Gamma}-q_{\Gamma'}+(\delta k_{\Gamma,\Gamma'})^s(q,q,q)=\bigl(h_{\Gamma}-h_{\Gamma'}|_{\Gamma}+\delta k_{\Gamma,\Gamma'}\bigr)^s(q,q,q) \qquad=0.
\end{align*}
Since $s$ is invertible, $d(t_{\Gamma,\Gamma',X})=0$, as required.

For naturality, take a closed degree-zero matrix $A:X=(V,q_X)\to Y=(W,q_Y)$.
Let $d$ denote the differential on
$\mathcal{H}om_{\operatorname{Tw}(\Lambda_-|_{\Gamma'})}(F_{\Gamma}X,F_{\Gamma'}Y)$.
Then 
\begin{align*}
&t_{\Gamma,\Gamma',Y}\circ\operatorname{Tw}(f_{\Gamma})_1(A)
 -\operatorname{Tw}(f_{\Gamma'})_1(A)\circ t_{\Gamma,\Gamma',X}\\
\overset{(1)}={}&\bigl(1_W-s^{-1}k_{\Gamma,\Gamma'}^s(q_Y,q_Y)\bigr)
 \circ s^{-1}\!\left(
 sA+\sum_{r+t=2}h_{\Gamma}^s(q_Y^{\otimes r},sA,q_X^{\otimes t})
 \right)\\
&\quad-s^{-1}\!\left(
 sA+\sum_{r+t=2}(h_{\Gamma'}|_{\Gamma})^s(q_Y^{\otimes r},sA,q_X^{\otimes t})
 \right)\circ\bigl(1_V-s^{-1}k_{\Gamma,\Gamma'}^s(q_X,q_X)\bigr)\\
\overset{(2)}={}&s^{-1}\!\left(
 \sum_{r+t=2}(h_{\Gamma}-h_{\Gamma'}|_{\Gamma})^s(q_Y^{\otimes r},sA,q_X^{\otimes t})
 -b_2\bigl(k_{\Gamma,\Gamma'}^s(q_Y,q_Y),sA\bigr)
 +b_2\bigl(sA,k_{\Gamma,\Gamma'}^s(q_X,q_X)\bigr)
 \right)\\
\overset{(3)}={}&s^{-1}\!\left(
 -\sum_{r+t=2}(\delta k_{\Gamma,\Gamma'})^s(q_Y^{\otimes r},sA,q_X^{\otimes t})
 -b_2\bigl(k_{\Gamma,\Gamma'}^s(q_Y,q_Y),sA\bigr)
 +b_2\bigl(sA,k_{\Gamma,\Gamma'}^s(q_X,q_X)\bigr)
 \right)\\
\overset{(4)}={}&s^{-1}\Bigl(
b_2\bigl(q_Y,k_{\Gamma,\Gamma'}^s(q_Y,sA)+k_{\Gamma,\Gamma'}^s(sA,q_X)\bigr) +b_2\bigl(k_{\Gamma,\Gamma'}^s(q_Y,sA)+k_{\Gamma,\Gamma'}^s(sA,q_X),q_X\bigr)
 \Bigr)\\
\overset{(5)}={}&d\left(s^{-1}\bigl(k_{\Gamma,\Gamma'}^s(q_Y,sA)+k_{\Gamma,\Gamma'}^s(sA,q_X)\bigr)\right).
\end{align*}
Here $(1)$ is from the definition, $(2)$ uses $E^2=0$, $(3)$ uses $h_{\Gamma}-h_{\Gamma'}|_{\Gamma}=-\delta k_{\Gamma,\Gamma'}$, $(4)$ uses
\[
\begin{aligned}
(\delta k_{\Gamma,\Gamma'})^s(u,v,w)
={}&-b_2\bigl(k_{\Gamma,\Gamma'}^s(u,v),w\bigr)
-(-1)^{|u|}b_2\bigl(u,k_{\Gamma,\Gamma'}^s(v,w)\bigr)\\
&-k_{\Gamma,\Gamma'}^s\bigl(b_2(u,v),w\bigr)
-(-1)^{|u|}k_{\Gamma,\Gamma'}^s\bigl(u,b_2(v,w)\bigr),
\end{aligned}
\]
$(5)$ follows from the twisted differential formula and $E^2=0$.
Thus, in $H^0\operatorname{Tw}(\Lambda_-|_{\Gamma'})$,
\[
(\vartheta_{\Gamma,\Gamma'})_Y\circ F_{\Gamma}([A])
=F_{\Gamma'}([A])\circ(\vartheta_{\Gamma,\Gamma'})_X.
\]

$\vartheta_{\Gamma,\Gamma'}$ is a natural isomorphism, where $1_V+s^{-1}k_{\Gamma,\Gamma'}^s(q,q)$ represents the inverse.

Finally, $\vartheta_{\Gamma,\Gamma}=1$ follows from the construction.
It remains to show for finite subsets $\Gamma \subseteq \Gamma' \subseteq \Gamma'' (\subseteq \omega_3)$, one has $\vartheta_{\Gamma,\Gamma''}=(\vartheta_{\Gamma',\Gamma''} \iota_{\Gamma,\Gamma'}^+)\circ(\iota_{\Gamma',\Gamma''}^-\vartheta_{\Gamma,\Gamma'})$.
Note that
\[
\delta\bigl(k_{\Gamma,\Gamma''}-k_{\Gamma,\Gamma'}-k_{\Gamma',\Gamma''}|_{\Gamma}\bigr)
=(h_{\Gamma''}|_{\Gamma}-h_{\Gamma})-(h_{\Gamma'}|_{\Gamma}-h_{\Gamma})-(h_{\Gamma''}|_{\Gamma}-h_{\Gamma'}|_{\Gamma})=0.
\]
By Lemma~\ref{lem_finite_ordinal_acyclicity}, take $l\in C_R^1(\mathcal C_{\Gamma},E_{\Gamma})$ with $\delta l=k_{\Gamma,\Gamma''}-k_{\Gamma,\Gamma'}-k_{\Gamma',\Gamma''}|_{\Gamma}$. 
Fix $X=(V,q)$.
Again we put
$q_\star=q+h_\star^s(q,q,q)$.
Let $d$ denote the differential on
$\mathcal{H}om_{\operatorname{Tw}(\Lambda_-|_{\Gamma''})}
((V,q_{\Gamma}),(V,q_{\Gamma''}))$.
Write $l^s s=sl$.
Suppressing the natural inclusions, we have
\begin{align*}
&t_{\Gamma,\Gamma'',X}-t_{\Gamma',\Gamma'',X}\circ t_{\Gamma,\Gamma',X}\\
\overset{[1]}={}&1_V-s^{-1}k_{\Gamma,\Gamma''}^s(q,q)
 -\bigl(1_V-s^{-1}(k_{\Gamma',\Gamma''}|_{\Gamma})^s(q,q)\bigr)
  \circ\bigl(1_V-s^{-1}k_{\Gamma,\Gamma'}^s(q,q)\bigr)\\
\overset{[2]}={}&-s^{-1}\bigl(k_{\Gamma,\Gamma''}-k_{\Gamma,\Gamma'}-k_{\Gamma',\Gamma''}|_{\Gamma}\bigr)^s(q,q)
\quad \overset{[3]}={}-s^{-1}(\delta l)^s(q,q)\\
\overset{[4]}={}&-s^{-1}\Bigl(
 b_2\bigl(l^s(q),q\bigr)
 +b_2\bigl(q,l^s(q)\bigr)
 -l^s\bigl(b_2(q,q)\bigr)\Bigr)
\overset{[5]}={}-s^{-1}\Bigl(
 b_2\bigl(l^s(q),q\bigr)
 +b_2\bigl(q,l^s(q)\bigr)\Bigr)\\
\overset{[6]}={}&s^{-1}\Bigl(
 b_2\bigl(q_{\Gamma''},-l^s(q)\bigr)
 +b_2\bigl(-l^s(q),q_{\Gamma}\bigr)\Bigr)
\quad \overset{[7]}={}d\bigl(-s^{-1}l^s(q)\bigr).
\end{align*}
Here $[2]$ and $[6]$ use $E^2=0$, $[5]$ uses $l^s(b_2(q,q))=-l^s(b_4^+(q,q,q,q))=0$, and the remaining follows from definitions or constructions.
Taking $H^0$ gives the result.
\end{proof}

\begin{remark}
    Some verifications in Lemma~\ref{lem_twisted_overlap_coherence} are not necessary by \cite[Corollary 11.22]{BLMPretriangulated}.
\end{remark}

\begin{proposition}\label{prop_marked_core_descent}
There is a triangle equivalence $\operatorname{Perf}(\Lambda_+)\simeq\operatorname{Perf}(\Lambda_-)$ fixing the one-term objects and their graded morphism groups, i.e.,
\[
F_{\Gamma}(\alpha,0)=(\alpha,0), \quad 
\operatorname{Hom}_{H^0\operatorname{Tw}(\Lambda_+|_{\Gamma})}
((\alpha,0),(\beta,0)[d])
\xlongrightarrow{\sim}
\operatorname{Hom}_{H^0\operatorname{Tw}(\Lambda_-|_{\Gamma})}
((\alpha,0),(\beta,0)[d]).
\]
\end{proposition}

\begin{proof}
We first construct a triangle equivalence $F:H^0\operatorname{Tw}(\Lambda_+)
\longrightarrow H^0\operatorname{Tw}(\Lambda_-)$.

For each twisted complex $X$, let $\Gamma_X$ be the finite set of
ordinals occurring in its underlying formal sum, taking
$\Gamma_X=\{0\}$ for the empty sum, and put $F(X)=F_{\Gamma_X}(X)$.
For a morphism $a:X\to Y$, choose a finite set
$\Gamma'\supseteq \Gamma_X\cup \Gamma_Y$ and define $F(a)=(\vartheta_{\Gamma_Y,\Gamma'})_Y^{-1}
\circ F_{\Gamma'}(a)\circ(\vartheta_{\Gamma_X,\Gamma'})_X$.
If $\Gamma'\subseteq \Gamma''$, then Lemma~\ref{lem_twisted_overlap_coherence} gives (suppressing natural inclusions)
\begin{align*}
(\vartheta_{\Gamma_Y,\Gamma''})_Y^{-1}\circ F_{\Gamma''}(a)\circ(\vartheta_{\Gamma_X,\Gamma''})_X=(\vartheta_{\Gamma_Y,\Gamma'})_Y^{-1}\circ(\vartheta_{\Gamma',\Gamma''})_Y^{-1}\circ F_{\Gamma''}(a)\circ(\vartheta_{\Gamma',\Gamma''})_X\circ(\vartheta_{\Gamma_X,\Gamma'})_X=(\vartheta_{\Gamma_Y,\Gamma'})_Y^{-1}\circ F_{\Gamma'}(a)
 \circ(\vartheta_{\Gamma_X,\Gamma'})_X.
\end{align*}
Hence, $F(a)$ is well-defined.
Similarly $F$ preserves composition, identities and $R$-linearity.

For fixed $X,Y$, take $\Gamma_X$, $\Gamma_Y$ and $\Gamma'$ as above.
Then $F$ acts on hom sets as the composite
\[
\operatorname{Hom}_{H^0\operatorname{Tw}(\Lambda_+)}(X,Y)
\xrightarrow[\sim]{\ F_{\Gamma'}\ }
\operatorname{Hom}_{H^0\operatorname{Tw}(\Lambda_-)}(F_{\Gamma'}X,F_{\Gamma'}Y)\xrightarrow{\sim}
\operatorname{Hom}_{H^0\operatorname{Tw}(\Lambda_-)}(FX,FY).
\]
Hence $F$ is fully faithful by
Lemma~\ref{lem_finite_twisted_equivalence}.
For any $Y\in H^0\operatorname{Tw}(\Lambda_-)$, choose finite
$\Gamma'\supseteq \Gamma_Y\cup\{0\}$.
The same lemma gives $X\in H^0\operatorname{Tw}(\Lambda_+|_{\Gamma'})$
with $F_{\Gamma'}(X)\cong Y$.
Since $(\vartheta_{\Gamma_X,\Gamma'})_X:F(X)\xrightarrow{\sim}F_{\Gamma'}(X)$,
the functor $F$ is essentially surjective.

Since the axioms of a triangulated category are checked in finite stages, $F$ is an equivalence. 
Idempotent completion passes the isomorphism to perfect categories. 
\end{proof}

\subsection{Detecting differences in \texorpdfstring{$m_4$}{m4} via Toda brackets}\label{sec_detector_construction}
\label{subsec_large_branches}

The category $\operatorname{Tw}(\Lambda_\pm)$ uses finite formal coproducts. In this subsection, we allow formal coproducts of $\leq \aleph_3$ representable right modules, obtaining an inclusion $\operatorname{Tw}(\Lambda_\pm) \hookrightarrow \operatorname{Tw}\bigl(\Lambda_\pm^{(\leq\aleph_3)}\bigr)$.
The resulting larger thick subcategories of the derived categories detect the difference between the two $m_4$ operations.

Let $\Lambda_b^{(\leq\aleph_3)}$ have objects families $(\alpha_i)_{i\in L}$ of ordinals below $\omega_3$, with $|L|\leq|\omega_3|$.
The family represents the free right $\mathcal C$-module $\coprod_{i\in L}y_{\alpha_i}$.
For any such two families, set
\[
\mathcal{H}om_{\Lambda_b^{(\leq\aleph_3)}}((\alpha_i)_{i\in L},(\beta_j)_{j\in K})^r
 =\prod_{i\in L}\coprod_{j\in K}\mathcal{H}om_{\Lambda_b}(\alpha_i,\beta_j)^r.
\]
Thus a morphism is a column-finite matrix.
Hence one can extens $m_i$ by summing over composable matrix entries. 
From a fixed input column, only finitely many paths of a fixed length occur, and these operations remain column-finite.
The Stasheff identities are follow entrywise from Lemma~\ref{lem_square_zero_detector}.
For the corresponding free modules $P,Q$, one has
\[
\mathcal{H}om_{\Lambda_b^{(\leq\aleph_3)}}(P,Q)^0=\operatorname{Hom}_{\mathcal C}(P,Q), \qquad \mathcal{H}om_{\Lambda_b^{(\leq\aleph_3)}}(P,Q)^{-2}=\operatorname{Hom}_{\mathcal C}(P,Q\otimes_{\mathcal C}E).
\]

For $b \in \{+, -\}$, this gives an $A _\infty$-functor $\mathbb Y_b : \Lambda_b^{(\leq \aleph_3)} \longrightarrow \mathrm{Mod}_{\infty} (\Lambda_b)$ via
\[
(\alpha_i)_{i \in L} \longmapsto\mathcal{H}om_{\Lambda_b^{(\leq\aleph_3)}}(\iota_b(-),(\alpha_i)_{i \in L})= \coprod_{i \in L}\mathcal{H}om_{\Lambda_b^{(\leq\aleph_3)}}(\iota_b(-),\alpha_i).
\]

\begin{lemma}\label{lem_large_matrix_realization}
$\mathbb Y_b$ induces a fully faithful triangle functor $H^0\operatorname{Tw}(\Lambda_b^{(\leq\aleph_3)})\longrightarrow D_\infty(\Lambda_b)$.
\end{lemma}

\begin{proof}
Let $\iota_b:\Lambda_b\hookrightarrow\Lambda_b^{(\leq\aleph_3)}$ be the full inclusion of singleton families.
As in \cite[\S7.5, Theorem]{KellerAinf}, extend the Yoneda functor to finite twisted complexes and restrict along $\iota_b$:
\[
\operatorname{Tw}(\Lambda_b^{(\leq\aleph_3)})
\longrightarrow\mathrm{Mod}_\infty(\Lambda_b^{(\leq\aleph_3)})
\xrightarrow{\iota_b^*}\mathrm{Mod}_\infty(\Lambda_b).
\]
Note that \cite[\S7.5]{KellerAinf} works over fields; over the commutative ring $R$, the extension is given by the same finite twisting formulas; see \cite[\S11.20 and Proposition~13.19]{BLMPretriangulated}.
Taking $H^0$ and then localising at objectwise quasi-isomorphisms gives a triangle functor
\[
\widetilde{\mathbb Y}_b:
H^0\operatorname{Tw}(\Lambda_b^{(\leq\aleph_3)})
\longrightarrow D_\infty(\Lambda_b)
\]
whose restriction to matrix objects is induced by $\mathbb Y_b$.

By \cite[Appendix~A, Theorem~A.1]{LyubashenkoManzyuk} one has a homotopy equivalence
\[
\mathcal{H}om_{\mathrm{Mod}_\infty(\Lambda_b)}
\bigl(\mathcal{H}om_{\Lambda_b}(-,\alpha),N\bigr)
\simeq N(\alpha).
\]
For $P=(\alpha_i)_{i\in L}$ and a module $N$, it gives $\mathcal{H}om_{\mathrm{Mod}_\infty(\Lambda_b)}(\mathbb Y_b(P),N) \simeq\prod_{i\in L}N(\alpha_i)$.
Hence
\[
\operatorname{Hom}_{D_\infty(\Lambda_b)}
(\mathbb Y_b(P),\mathbb Y_b(Q)[d])
\cong\prod_{i\in L}H^d(\mathbb Y_b(Q)(\alpha_i))=\prod_{i\in L}\coprod_{j\in K}
\mathcal{H}om_{\Lambda_b}(\alpha_i,\beta_j)^d
=\mathcal{H}om_{\Lambda_b^{(\leq\aleph_3)}}(P,Q)^d.
\]
The inverse isomorphism is induced by $(\mathbb Y_b)_1$ and minimality:
\[
\begin{aligned}
&\mathcal{H}om_{\Lambda_b^{(\leq\aleph_3)}}(P,Q)^d
\xlongequal{m_1=0}
H^d\!\left(\mathcal{H}om_{\Lambda_b^{(\leq\aleph_3)}}(P,Q)\right)\\
&\quad \xrightarrow{\ H^d((\mathbb Y_b)_1)\ }
H^d\!\left(\mathcal{H}om_{\mathrm{Mod}_\infty(\Lambda_b)}
(\mathbb Y_b(P),\mathbb Y_b(Q))\right)\xlongrightarrow{\sim}
\operatorname{Hom}_{D_\infty(\Lambda_b)}
(\mathbb Y_b(P),\mathbb Y_b(Q)[d]).
\end{aligned}
\]
Thus $\widetilde{\mathbb Y}_b$ is fully faithful on matrix objects and their shifts, i.e., $P[r]$'s.
Taking shifts and cones (by five lemma) proves full faithfulness on all finite twisted complexes.
\end{proof}

Let $\mathcal V_b$ denote the image of $\widetilde{\mathbb Y}_b$.
By Lemma~\ref{lem_large_matrix_realization}, one has fully faithful inclusion of thick subcategories $\operatorname{Perf}(\Lambda_b) \subseteq \mathcal V_b \subseteq D_\infty (\Lambda_b)$, and all three categories are algebraic over $R$.

\medskip

Note that each term ($P_d$'s) in the projective resolution of $\underline R$ (Lemma~\ref{lem_tail_simplex_exactness}) are in $\mathcal V_b$, by Lemma~\ref{lem_large_matrix_realization}.
The graded Hom groups are computed as
\[
\operatorname{Hom}_{\mathcal V_b}(P_p,P_q[d])
\cong
\begin{cases}
\operatorname{Hom}_{\mathcal C}(P_p,P_q),&d=0,\\
\operatorname{Hom}_{\mathcal C}(P_p,P_q\otimes_{\mathcal C}E),&d=-2,\\
0,&\text{else}.
\end{cases}
\]
Hence the full subcategory on $P_p$'s is $2$-split.

\begin{remark}
    The computation for $d=-2$ is given by $y_\beta \otimes _{\mathcal C} E \cong E(-,\beta)$ and coproducts.
\end{remark}

Put $\mathcal I=(d_1\otimes_{\mathcal C}E)_\ast \operatorname{Hom}_{\mathcal C}(P_4,P_1\otimes_{\mathcal C}E) +d_4{}^\ast \operatorname{Hom}_{\mathcal C}(P_3,P_0\otimes_{\mathcal C}E) \subseteq \operatorname{Hom}_{\mathcal C}(P_4,P_0\otimes_{\mathcal C}E)$.

\begin{proposition}\label{prop_branch_toda_cosets}
With conventions of Definitions~\ref{def_twisted_complex} and~\ref{def_fourfold_bracket}, the Toda brackets are
\[
\langle d_1,d_2,d_3,d_4\rangle_{\mathcal V_-}=\mathcal I,
\qquad
\langle d_1,d_2,d_3,d_4\rangle_{\mathcal V_+}
=-m_4^+(d_1,d_2,d_3,d_4)+\mathcal I.
\]
\end{proposition}

\begin{proof}
Consider $P_4 \xrightarrow{d_4}P_3 \xrightarrow{d_3}P_2 \xrightarrow{d_2}P_1 \xrightarrow{d_1}P_0$.
Since the terms belong to a $2$-split subcategory, Proposition~\ref{prop_fourfold_indeterminacy} gives a nonempty Toda coset with indeterminacy $\mathcal I$.

Fix $b \in \{+, -\}$.
We work in $\operatorname{Tw}(\Lambda_b^{(\leq\aleph_3)})$.
By definition, $\langle d_1,d_2,d_3,d_4\rangle$ consists of the morphism $\gamma:P_4[2]\to P_0$ for which there exist objects $C,Y$ and morphisms satisfying the certain conditions.
Take
\[
C=\left(P_1\oplus P_2[1],
q_C = \begin{psmallmatrix}0&sd_2\\0&0\end{psmallmatrix}\right),
\qquad
Y=\left(P_1\oplus P_2[1]\oplus P_3[2],
q_Y = \begin{psmallmatrix}0&sd_2&0\\0&0&sd_3\\0&0&0\end{psmallmatrix}\right).
\]
    \begin{enumerate}
        \item The morphism $\gamma$ is defined as the composite $P_4[2]\xlongrightarrow{\alpha=\begin{psmallmatrix}0\\0\\d_4[2]\end{psmallmatrix}} Y\xlongrightarrow{\beta=\begin{psmallmatrix}d_1&0&0\end{psmallmatrix}} P_0$.
        \item The following are distinguished triangles:
              \[
                  P_1\xlongrightarrow{u=\begin{psmallmatrix}1\\0\end{psmallmatrix}}C\xlongrightarrow{q=\begin{psmallmatrix}0&1\end{psmallmatrix}}P_2[1]\xlongrightarrow{d_2[1]}P_1[1],\qquad
                  C\xlongrightarrow{v=\begin{psmallmatrix}1&0\\0&1\\0&0\end{psmallmatrix}}Y\xlongrightarrow{p=\begin{psmallmatrix}0&0&1\end{psmallmatrix}}P_3[2]\xlongrightarrow{\delta=\begin{psmallmatrix}0\\ d_3[2]\end{psmallmatrix}}C[1].
              \]
        \item The following three diagrams commute:
              \[
                  \xymatrix@C=3em@R=1.5em{
                      P_3[2] \ar[r]^-\delta \ar[dr]_-{d_3[2]}
                      & C[1] \ar[d]^-{q[1]}
                      & P_4[2] \ar[r]^-\alpha \ar[dr]_-{d_4[2]}
                      & Y \ar[d]^-p
                      & P_1 \ar[r]^-{v\circ u} \ar[dr]_-{d_1}
                      & Y \ar[d]^-\beta \\
                      & P_2[2] & & P_3[2] & & P_0
                  }
              \]
    \end{enumerate}
    The above are easily checked.
    Hence we regard $\langle d_1,d_2,d_3,d_4\rangle$ as a subset of $\operatorname{Hom}_{\mathcal V_b}(P_4,P_0[-2])$.
It remains to show that $-m_4^b(d_1,d_2,d_3,d_4)
\in\langle d_1,d_2,d_3,d_4\rangle_{\mathcal V_b}$.
Note that
\begin{align*}
    [\gamma] & = [\beta] \circ [\alpha] = [m_2^{\operatorname{Tw}}(\beta, \alpha)] = s^{-1} b_2^{\operatorname{Tw}}(s \beta, s\alpha) =s^{-1}\sum_{j\geq0}b_{j+2}(s\beta,q_Y^{\otimes j},s\alpha)\\
    &=s^{-1}\bigl(b_2(s\beta,s\alpha)
+b_4(s\beta,q_Y,q_Y,s\alpha)\bigr)
\overset{\langle 1\rangle}=s^{-1}b_4\bigl(sd_1,sd_2,sd_3,s(d_4[2])\bigr)\\
& = -(m_4^b)^\Sigma (d_1, d_2,d_3,d_4[2]) \overset{\langle 2\rangle}=-m_4^b(d_1,d_2,d_3,d_4[2]).
\end{align*}
For $\langle 1\rangle$, expanding the matrix gives $b_2(s\beta,s\alpha)
=b_2(sd_1,0)+b_2(0,0)+b_2(0,s(d_4[2]))=0$ and and the unique path contributing to the matrix expansion of $b_4(s\beta,q_Y,q_Y,s\alpha)$ is $[sd_1,sd_2,sd_3,s(d_4[2])]$.
The Koszul rule for $m^\Sigma$ in the Preliminaries gives $\langle 2\rangle$.

Finally Lemma~\ref{lem_large_matrix_realization} and Remark~\ref{rem_toda_functoriality} transport these Toda cosets to $\mathcal V_b$.
Since $m_4^-=0$, this proves the two asserted coset formulas.
\end{proof}

\begin{proposition}\label{prop_corrected_toda_endpoint}
For every $f\in\langle d_1,d_2,d_3,d_4\rangle_{\mathcal V_b}$, the composite
$(d_0\otimes_{\mathcal C}E)\circ f:P_4 \to P_0\otimes_{\mathcal C} E\to \underline R \otimes_{\mathcal C} E = A$ is a cocycle and
\[
\bigl[(d_0\otimes_{\mathcal C}E)\circ f\bigr]
=\begin{cases}
0,&b=-,\\
-[c] (\neq 0),&b=+,
\end{cases}
\quad\text{in }\operatorname{Ext}_{\mathcal C}^4(\underline R,A).
\]
\end{proposition}

\begin{proof}
Note that for ordinals $\alpha_0 < \alpha_1 < \alpha_2 < \alpha_3 < \alpha _4 <\omega_3$, one has
\[
m_4^+(u_{\alpha_3,\alpha_4},u_{\alpha_2,\alpha_3},u_{\alpha_1,\alpha_2},u_{\alpha_0,\alpha_1}) = c_{\alpha_0}[\alpha_0, \ldots, \alpha_4].
\]
Indeed, since $m_4^+(\ldots, 1, \ldots) = 0$, the only surviving path from $\sigma = [\alpha_0, \ldots, \alpha_4]$ to $\varnothing$ is
\[
[\alpha_0, \ldots, \alpha_4] \to [\alpha_1, \ldots, \alpha_4] \to 
[\alpha_2, \alpha_3, \alpha_4] \to 
[\alpha_3, \alpha_4] \to 
[\alpha_4] \to \varnothing.
\]
Take the canonical inclusion $\iota_{\alpha_4}:A(\alpha_0)
\to
(P_0\otimes_{\mathcal C}E)(\alpha_0)
\cong\coprod_{\beta<\omega_3}A(\alpha_0)$.
Then the following composite is identical
\[
(d_0 \otimes_{\mathcal C} E) \circ \iota_{\alpha_4} : A(\alpha_0) \longrightarrow (P_0 \otimes _{\mathcal C} E)(\alpha_0) \longrightarrow (\underline{R} \otimes_{\mathcal C} E)(\alpha_0)=A(\alpha_0).
\]
It follows that 
\[
\bigl((d_0\otimes_{\mathcal C}E)\circ (m_4^+(d_1,d_2,d_3,d_4))\bigr)_{\alpha_0}(\sigma)=
(d_0\otimes_{\mathcal C}E)_{\alpha_0} \bigl(\iota_{\alpha_4}
\bigl(c_{\alpha_0}(\sigma)\bigr)\bigr)=c_{\alpha_0}(\sigma)
\]
Since the above simplex summands generate $P_4$, we obtain
\[
(d_0\otimes_{\mathcal C}E)\circ m_4^+(d_1,d_2,d_3,d_4)=c,
\qquad c\circ d_5=0.
\]
Thus the composite give the classes $0$ and $-[c]$, respectively.

We then show the cohomology class is independent of the representative in Toda bracket.
Take another $f' \in\langle d_1,d_2,d_3,d_4\rangle_{\mathcal V_b}$.
Then by the construction of $\mathcal I$ in Proposition~\ref{prop_branch_toda_cosets}, take $f-f' = (d_1\otimes_{\mathcal C}E)a+v d_4$ with $a:P_4\to P_1\otimes_{\mathcal C}E$ and $v:P_3\to P_0\otimes_{\mathcal C}E$.
Thus
\[
(d_0\otimes_{\mathcal C}E)\circ\bigl((d_1\otimes_{\mathcal C}E)\circ a+v\circ d_4\bigr)
=\bigl((d_0\otimes_{\mathcal C}E)\circ v\bigr)\circ d_4.
\]
The right-hand side is a degree-four boundary in $\mathcal{H}om_{\mathcal C}(P_\bullet,A)$, thus is killed by $[-]$.

Finally, $[c]\ne0$ follows from Lemma~\ref{lem_ordinal_nonvanishing}.
\end{proof}

\subsection{Construct \texorpdfstring{$\mathcal T$}{T} from
\texorpdfstring{$\mathcal V_\pm$}{Vpm}}
\label{subsec_function_coefficients}

We retain the cardinal $\theta$ in the outline and put
\[
I=\theta\times\{+,-\},\qquad
\operatorname{Sign}(\xi,+)=+,\qquad
\operatorname{Sign}(\xi,-)=-.
\]
We construct $\mathcal T$ by adjoining triangulated categories $\mathcal T_J$'s with a common perfect core:
\begin{equation}\label{eq_TJ}
    \mathcal T_J
=\left(\prod_{i\in J}\mathcal V_{\operatorname{Sign}(i)}\right)
\times\operatorname{Perf}(\Lambda_-(I\setminus J)), \qquad |J|<\aleph_0
\end{equation}

For $V \in \mathrm{Mod}_R$, define the set of ``almost constant'' maps as follows
\[
\mathsf L(V) =\left\{ f:I\to V\ \middle|\ |f(I)|<\aleph_0,\quad \exists S\subseteq I,\ |S|<\theta,\quad f|_{I\setminus S} \ \text{is constant}.\right\}.
\]
A subset $U\subseteq I$ is \emph{admissible} if
$|U|<\theta$ or $|I\setminus U|<\theta$.
For admissible $U\subseteq I$, define
\[
\mathsf L_U(V)=\{f|_U\mid f\in\mathsf L(V)\}.
\]
We write $\operatorname{ev}_i$ for evaluation at $i\in I$.
For functions in $\mathsf L(V)$, we write $\operatorname{ev}_\infty$
for their unique common value outside a subset of
cardinality less than $\theta$.

\begin{lemma}\label{lem_function_coefficient_splitting}
Let $U\subseteq I$ be admissible and let $V$ be an $R$-module.
\begin{enumerate}
\item Pointwise operations make $\mathsf L_U(V)$ an
$R$-module.
\item The following canonical map is an isomorphism
\[
\mathsf L_U(R)\otimes_R V\longrightarrow\mathsf L_U(V),
\qquad a\otimes v\longmapsto(i\mapsto a(i)v).
\]
\item A finite partition $U=\coprod_{r=1}^nU_r$ of admissible subsets induces an isomorphism
\[
\mathsf L_U(V)\xlongrightarrow{\sim}
\bigoplus_{r=1}^n\mathsf L_{U_r}(V), \qquad f \longmapsto (f|_{U_1}, f|_{U_2}, \ldots, f|_{U_n}).
\]
\end{enumerate}
\end{lemma}

\begin{proof}
Straightforward.
\end{proof}

\begin{definition}\label{def_function_coefficient_core}
For admissible $U\subseteq I$, let $\Lambda_-(U)$ be the
dg category with $\operatorname{Ob}=\omega_3$ and
\[
\mathcal{H}om_{\Lambda_-(U)}(\alpha,\beta)^d
=\mathsf L_U\!\left(\mathcal{H}om_{\mathcal B}(\alpha,\beta)^d
\right).
\]
Its differential is zero and its composition is pointwise.
Evaluation and restriction act on coefficients and leave
the object labels unchanged.
\end{definition}

\begin{lemma}\label{lem_function_core_splitting}
A finite partition
$U=\coprod_{r=1}^nU_r$ of admissible subsets induces a triangle equivalence by restriction
\[
\operatorname{Perf}(\Lambda_-(U))
\xrightarrow{\sim}
\prod_{r=1}^n\operatorname{Perf}(\Lambda_-(U_r)), 
\]
\end{lemma}

\begin{proof}
By Lemma~\ref{lem_function_coefficient_splitting}, the
morphism modules split according to the central idempotents
$1_{U_r}$.
After passing to perfect categories, these idempotents split
objects into their components on the parts $U_r$.
This gives the asserted equivalence.
\end{proof}

We identify $\operatorname{Perf}(\Lambda_-)
\xrightarrow{\sim}\operatorname{Perf}(\Lambda_+)$ and take the canonical fully faithful triangle functor $\iota_\pm:\operatorname{Perf}(\Lambda_\pm)
\hookrightarrow\mathcal V_\pm$.
For each finite subset $J\subseteq I$, define
\[
\mathcal T_J
=
\left(\prod_{i\in J}\mathcal V_{\operatorname{Sign}(i)}\right)
\times\operatorname{Perf}(\Lambda_-(I\setminus J)).
\]

\begin{proposition}\label{prop_finite_branch_transitions}
The categories $\mathcal T_J$, indexed by
finite subsets of $I$, form a directed system of
fully faithful triangle functors with thick images.
\end{proposition}

\begin{proof}
For finite subsets $J$ and $K$ with $J\subseteq K\subseteq I$, applying Lemma~\ref{lem_function_core_splitting} to the partition $I\setminus J
=\coprod_{k\in K\setminus J}\{k\}
\amalg(I\setminus K)$ gives an equivalence
\[
\operatorname{Perf}(\Lambda_-(I\setminus J))
\simeq
\left(\prod_{k\in K\setminus J}
\operatorname{Perf}(\Lambda_-)\right)
\times\operatorname{Perf}(\Lambda_-(I\setminus K)).
\]
Then use $\iota_{\operatorname{Sign}(k)}$ to get the transition functor $\Phi_{J,K}:\mathcal T_J
\to\mathcal T_K$.
Explicitly,
\[
\Phi_{J,K}\bigl((X_i)_{i\in J},P\bigr)
=
\left(
(X_i)_{i\in J},
\bigl(\iota_{\operatorname{Sign}(k)}
(\operatorname{ev}_kP)\bigr)_{k\in K\setminus J},
P|_{I\setminus K}
\right).
\]
$\Phi_{J,K}$ is fully faithful and has thick image, since these
properties hold for $\iota_\pm$.

For finite subsets $J$, $K$ and $L$ with $J\subseteq K\subseteq L$, one has $\Phi_{K,L}\Phi_{J,K} = \Phi_{J,L}$.
This equality should be an isomorphism; for convenience we identify the finite stages with their images in a common union model, so that the transition functors are inclusions.

The system is clearly directed.
\end{proof}

For infinite subset $J \subseteq I$, put $\mathcal T_J = \varinjlim_{K \subseteq J, \ |K| < \aleph_0}\mathcal T_K$.
Write $\mathcal T = \mathcal T_I$.

\subsection{Subcategories generated by fewer than \texorpdfstring{$\theta$}{theta} objects}\label{subsec_small_generated_algebraicity}

We show that every thick subcategory of $\mathcal T$ generated by fewer than $\theta$ objects is algebraic over $R$;
equivalently, $\mathcal T_J$ is algebraic for $|J| < \theta$.
We fix such $J$, and let $e_J : I \to R$ be map in $\mathsf L(R)$ by setting 
\[
e_J(i)=
\begin{cases}
1,&i\in J,\ \operatorname{Sign}(i)=+,\\
0,&\text{otherwise}.
\end{cases}
\]

\begin{lemma}
    There is an $A_\infty$ category $\Lambda_J$ whose underlying graded category and $m_{\neq 4}$ coincide with the dg category $\Lambda_-(I)$, and $\alpha_0\xrightarrow{f_1}\alpha_1
\xrightarrow{f_2}\alpha_2
\xrightarrow{f_3}\alpha_3
\xrightarrow{f_4}\alpha_4$ gives that
    \[
    m_4(f_4,f_3,f_2,f_1)(i) =e_J(i) \cdot m_4^+(f_4(i),f_3(i),f_2(i),f_1(i)).
    \]
\end{lemma}

\begin{proof}
    By construction, $m_4(f_4,f_3,f_2,f_1) \in \mathsf L(E(\alpha_0, \alpha_4)) = \mathsf L(A(\alpha_0))$.

The Stasheff identities are checked pointwise on $i \in I$, so they hold by Lemma~\ref{lem_square_zero_detector}.
\end{proof}

\begin{lemma}\label{lem_mixed_function_core}
There are triangle equivalences which are natural in admissible subset $U \subseteq I$ and $|J|<\theta$, denoted by $\Psi_{J,U}: \operatorname{Perf}(\Lambda_-(U))\xrightarrow{\sim}\operatorname{Perf}(\Lambda_J(U))$.
\end{lemma}

\begin{proof}
    Take $U_+=U\cap J\cap\operatorname{Sign}^{-1}(+)$ and $U_-=U\setminus U_+$. 
    Both parts are admissible. clearly
    \[
    \operatorname{Perf}(\Lambda_-(U))
    \xrightarrow{\sim}
    \operatorname{Perf}(\Lambda_-(U_+)) \times \operatorname{Perf}(\Lambda_-(U_-))
    \xrightarrow{\sim}
    \operatorname{Perf}(\Lambda_J(U_+)) \times \operatorname{Perf}(\Lambda_J(U_-)) 
    \xrightarrow{\sim}
    \operatorname{Perf}(\Lambda_J(U)) 
    \]
    are composite of isomorphisms natural in both arguments.
\end{proof}

The derived category $D_\infty(\Lambda_J)$ admits an $R$-linear dg enhancement \cite[\S13.35]{BLMPretriangulated} and \cite[Theorem~1.6.2]{Drinfeld}.
We then construct a fully faithful triangle inclusion $\mathcal T_J \hookrightarrow D_\infty(\Lambda_J)$.

\begin{proposition}\label{prop_small_stage_inclusion}
For $J\subseteq I$ ($|J|<\theta$), there is an
$R$-linear fully faithful triangle functor
\[
\Phi_J:\mathcal T_J\longrightarrow D_\infty(\Lambda_J).
\]
\end{proposition}

\begin{proof}
For admissible $U\subseteq I$, take the restriction $A_\infty$ functor $r_U:\Lambda_J\to\Lambda_J(U)$ which is identical on objects, and whose components are
\[
(r_U)_1:\mathsf L\bigl(\mathcal{H}om_{\mathcal B}(\alpha,\beta)^d\bigr)
\longrightarrow
\mathsf L_U\bigl(\mathcal{H}om_{\mathcal B}(\alpha,\beta)^d\bigr),
\quad f\longmapsto f|_U; \qquad (r_U)_{\geq 2} = 0.
\]
Note that the Stasheff identities are pointwise checked, in particular
\[
\bigl(m_4^{\Lambda_J}(f_4,\ldots,f_1)|_U\bigr)(u)
=e_J(u)m_4^+(f_4(u),\ldots,f_1(u))=m_4^{\Lambda_J(U)}(f_4|_U,\ldots,f_1|_U)(u).
\]

We define $\Phi_F : \mathcal T_F \to D_\infty (\Lambda_J)$ for finite subset $F \subseteq J$ first.
For $U = I \setminus F$, one has
\[
\mathcal T_F=
\left(\prod_{i\in F}\mathcal V_{\operatorname{Sign}(i)}\right)
\times\operatorname{Perf}(\Lambda_-(U)).
\]
\begin{enumerate}
    \item For each $i \in F \subseteq J$, take $\mathcal V_{\operatorname{Sign}(i)}\hookrightarrow D_\infty(\Lambda_{\operatorname{Sign}(i)}) = D_\infty(\Lambda_J(\{i\}))\xrightarrow{(r_{\{i\}})^*}D_\infty(\Lambda_J)$.
    \item For $\operatorname{Perf}(\Lambda_-(U))$, take $\operatorname{Perf}(\Lambda_-(U))\xrightarrow[\sim]{\Psi_{J,U}}\operatorname{Perf}(\Lambda_J(U))\hookrightarrow D_\infty(\Lambda_J(U))\xrightarrow{(r_U)^*}D_\infty(\Lambda_J)$.
\end{enumerate}
The fully faithful functors above yield another full faithful triangle functor
\[
\Phi_F:\mathcal T_F\longrightarrow D_\infty(\Lambda_J), \qquad
\bigl((X_i)_{i\in F},P\bigr) \longmapsto
\left(\bigoplus_{i\in F}r_{\{i\}}^*X_i\right)
\oplus (r_U)^*\Psi_{J,U}(P).
\]

Let $F\subseteq G\subseteq J$ be finite, then $I\setminus F=(G\setminus F)\amalg(I\setminus G)$ and thus yields a natural isomorphism $\operatorname{Perf}(\Lambda_-(I \setminus F)) \cong \operatorname{Perf}(\Lambda_-(G \setminus F)) \times \operatorname{Perf}(\Lambda_-(I \setminus G))$.
It is straightforward to check $\Phi_G\Phi_{F,G}\cong\Phi_F$.
These $\Phi_F$'s therefore induce a triangle functor
\[
\Phi_J:\mathcal T_J
=\varinjlim_{F\subseteq J,\ |F| < \aleph_0}
\mathcal T_F
\longrightarrow D_\infty(\Lambda_J).
\]
Any two objects lie in a common finite stage, where the functor is fully faithful. 
Hence $\Phi_J$ is fully faithful.
\end{proof}

\subsection{Almost constancy of characteristic values}\label{subsec_finite_large_characteristic}

We construct a map $\chi:
H^4_{\mathrm{cat}}(\mathcal F,\mathbb E|_{\mathcal F})\to\operatorname{Ext}_{\mathcal C}^4(\underline R,A)$ from the category cohomology to $\operatorname{Ext}^4_{\mathcal C}(\underline R, A)$, and show the map $\chi((\operatorname{ev}_{(\cdot)})_\ast z) : I\to \operatorname{Ext}_{\mathcal C}^4(\underline R,A)$ is constant except for a subset $S \subseteq I$ with cardinality $|S| < |I|$.

Let $\mathcal F$ be a skeleton  consisting of finite direct sums of representable functors ($y_\beta$'s).
Let $\mathcal P$ be an arbitrary fixed small full additive category of projective right $\mathcal C$-modules containing $\mathcal F$ and every $P_p$ in the resolution of $\underline{R}$.
Define the bifunctor
\[
\mathbb E:\mathrm{Mod}_{\mathcal C}^{\mathrm{op}}\times\mathrm{Mod}_{\mathcal C}
\longrightarrow\mathrm{Mod}_R,
\qquad
(P,Q)\longmapsto
\operatorname{Hom}_{\mathcal C}(P,Q\otimes_{\mathcal C}E),
\qquad \mathbb E(y_\alpha,y_\beta)\cong E(\alpha,\beta).
\]

\begin{lemma}\label{lem_category_cohomology_ext}
Let $\mathcal S$ be an arbitrary small category of projective right $\mathcal C$-modules (e.g. $\mathcal F$ or $\mathcal P$).
Take the inclusion $I_{\mathcal S} : \mathcal S \to \mathrm{Mod}_\mathcal C$.
Then
\[
H_{\mathrm{cat}}^\bullet\!\left(\mathcal S,\mathbb E|_{\mathcal S}\right)
\cong\operatorname{Ext}_{\operatorname{Fun}(\mathcal S,\mathrm{Mod}_\mathcal C)}^\bullet
(I_{\mathcal S},I_{\mathcal S}\otimes_{\mathcal C}E).
\]
\end{lemma}
\begin{proof}
The projective resolution of $I_\mathcal S$ in the functor category $\mathcal A = \operatorname{Fun}(\mathcal S, \mathrm{Mod}_{\mathcal C})$ is
\[
\cdots\longrightarrow{}
\coprod_{X_0\to X_1\to X_2}
X_0^{\coprod_{\operatorname{Hom}_{\mathcal S}(X_2,-)}}
\longrightarrow
\coprod_{X_0\to X_1}
X_0^{\coprod_{\operatorname{Hom}_{\mathcal S}(X_1,-)}}
\longrightarrow{}
\coprod_{X_0\in\mathcal S}
X_0^{\coprod_{\operatorname{Hom}_{\mathcal S}(X_0,-)}}
\longrightarrow I_{\mathcal S}\longrightarrow0.
\]
Applying $\operatorname{Hom}_{\mathcal A} (-,I_{\mathcal S}\otimes_{\mathcal C}E)$ gives a cochain.
The adjunction of universal coproduct and evaluation, $\left(P\longmapsto P^{\coprod_{\operatorname{Hom}_{\mathcal S}(Q,-)}}\right)\dashv\left(H\longmapsto H(Q)\right)$, gives the natural isomorphisms
\[
\begin{aligned}
&\operatorname{Hom}_{\mathcal A}\!\left(
\coprod_{X_0\to\cdots\to X_n}
X_0^{\coprod_{\operatorname{Hom}_{\mathcal S}(X_n,-)}},
I_{\mathcal S}\otimes_{\mathcal C}E\right) 
\cong \prod_{X_0\to\cdots\to X_n}\operatorname{Hom}_{\mathcal A}\!\left(
X_0^{\coprod_{\operatorname{Hom}_{\mathcal S}(X_n,-)}},
I_{\mathcal S}\otimes_{\mathcal C}E\right) \\
\cong{}&
\prod_{X_0\to\cdots\to X_n}
\operatorname{Hom}_{\mathcal C}
(X_0,X_n\otimes_{\mathcal C}E) = 
C^n_{\mathrm{cat}}\!\left(
\mathcal S,\mathbb E|_{\mathcal S}\right) \qquad (n\geq 0).
\end{aligned}
\]
Taking fourth cohomology yields the desired isomorphism.
\end{proof}

\begin{lemma}\label{lem_finite_free_restriction}
One has natural isomorphism $H_{\mathrm{cat}}^4\!\left(\mathcal P,\mathbb E|_{\mathcal P}\right) \cong H_{\mathrm{cat}}^4(\mathcal F,\mathbb E|_{\mathcal F})$.
\end{lemma}

\begin{proof}
    Let $Q_\bullet \to I_{\mathcal F}$ be the projective resolution of $I_{\mathcal F}$ in $\operatorname{Fun}(\mathcal F, \mathrm{Mod}_{\mathcal C})$.
    Put $k : \mathcal F \subseteq \mathcal P$.
    Then for each $X \in \mathcal P$, the left Kan extension gives
\[
(\operatorname{Lan}_k I_{\mathcal F})(X)
=\varinjlim_{(F\to X)\in\mathcal F/X}F\cong X, \qquad \text{i.e., $\operatorname{Lan}_kI_{\mathcal F}\cong I_{\mathcal P}$}.
\]
Here for $X \in \mathcal F$, the colimit is filtered, thus is exact.
Since $\operatorname{Lan}_k$ is exact and has an exact right adjoint, it preserves projective resolutions.
Hence $\operatorname{Lan}_kQ_\bullet$ is a projective resolution of $I_{\mathcal P}$.
Along with Lemma~\ref{lem_category_cohomology_ext}, one has natural isomorphisms
\begin{align*}
    &H^4_{\mathrm{cat}}\!\left(\mathcal P,
\mathbb E|_{\mathcal P}\right)\cong
\operatorname{Ext}^4_{\operatorname{Fun}(\mathcal P,\mathrm{Mod}_{\mathcal C})}
(I_{\mathcal P},I_{\mathcal P}\otimes_{\mathcal C}E)= H^4 \operatorname{Hom}_{\operatorname{Fun}(\mathcal P,\mathrm{Mod}_{\mathcal C})}
(\operatorname{Lan}_k Q_\bullet,I_{\mathcal P}\otimes_{\mathcal C}E)\\
\cong {}&H^4 \operatorname{Hom}_{\operatorname{Fun}(\mathcal F,\mathrm{Mod}_{\mathcal C})}
(Q_\bullet,I_{\mathcal F}\otimes_{\mathcal C}E)
\cong
\operatorname{Ext}^4_{\operatorname{Fun}(\mathcal F,\mathrm{Mod}_{\mathcal C})}
(I_{\mathcal F},I_{\mathcal F}\otimes_{\mathcal C}E)
\cong H^4_{\mathrm{cat}}(\mathcal F,\mathbb E|_{\mathcal F}).\qedhere
\end{align*}
\end{proof}

Let $\omega_3$ be the ordinary ordinal category, viewed as a
subcategory of $\mathcal F$ via $\beta\mapsto y_\beta$.

\begin{lemma}
    There are natural isomorphisms $H^n_{\mathrm{cat}}(\omega_3,E|_{\omega_3}) \cong  \operatorname{Ext}^n_{\mathcal C}(\underline R,A)$ for $n \geq 0$.
\end{lemma}

\begin{proof}
Induced by restriction $\omega_3 \to \mathcal F$, via $\beta \mapsto y_\beta$. 
A cochain $u \in C_{\mathrm{cat}}^n(\mathcal F, \mathbb E|_{\mathcal F})$ gives
\[
\operatorname{res}(u) : [\beta_0 \to \beta_1 \to \cdots \to \beta_n] \mapsto u[y_{\beta_0} \to y_{\beta_1} \to \cdots \to y_{\beta_n}].
\]
By construction, $(C_{\mathrm{cat}}^\bullet(\omega_3, E|_{\omega_3}) , \delta)$ and $\operatorname{Hom}_{\mathcal C}(P_\bullet ,A)$ has isomorphic terms:
\[
C^n_{\mathrm{cat}}(\omega_3,E|_{\omega_3})
=\prod_{\beta_0\leq\cdots\leq\beta_n<\omega_3}
E(\beta_0,\beta_n)
=\prod_{\beta_0\leq\cdots\leq\beta_n<\omega_3}
A(\beta_0) \cong \operatorname{Hom}_{\mathcal C}
\left(\coprod_{\beta_0\leq\cdots\leq\beta_n}y_{\beta_0},A\right).
\]
The diffential $\delta$ is constructed by bar resolution, so $(C_{\mathrm{cat}}^\bullet(\omega_3, E|_{\omega_3}) , \delta)\cong \operatorname{Hom}_{\mathcal C}(P_\bullet ,A)$.
Taking cohomology gives $H^n_{\mathrm{cat}}(\omega_3,E|_{\omega_3}) \cong  \operatorname{Ext}^n_{\mathcal C}(\underline R,A)$.
\end{proof}

\begin{definition}\label{def_characteristic_map}
Define $\chi$ as the composite $H^4_{\mathrm{cat}}(\mathcal F,\mathbb E|_{\mathcal F}) \xrightarrow{\operatorname{res}}H^4_{\mathrm{cat}}(\omega_3,E|_{\omega_3}) \xrightarrow{\sim} \operatorname{Ext}^4_{\mathcal C}(\underline R,A)$.
\end{definition}

\begin{lemma}
\label{lem_finite_large_characteristic}
For every zero-normalized cocycle $u\in Z^4_{\mathrm{cat}}\!\left( \mathcal P, \mathbb E|_{\mathcal P}
\right)$, the composite $P_4\xrightarrow{u(d_1,d_2,d_3,d_4)}
P_0\otimes_{\mathcal C}E
\xrightarrow{d_0\otimes_{\mathcal C}E}A$
represents $\chi([u|_{\mathcal F}])$ in
$\operatorname{Ext}^4_{\mathcal C}(\underline R,A)$.
\end{lemma}

\begin{proof}
Let $\mathcal P^+$ be the additive subcategory generated by $\mathcal P$ and $\underline R$.
Write the inclusions $\mathcal F\xrightarrow{k}\mathcal P\xrightarrow{h}\mathcal P^+$ and $j=hk$.
For a projective resolution $Q_\bullet\to I_{\mathcal F}$, Lemma~\ref{lem_finite_free_restriction} shows that $\operatorname{Lan}_kQ_\bullet$ is a projective resolution of $I_{\mathcal P}$.
Similarly, since
\[
(\operatorname{Lan}_jI_{\mathcal F})(\underline R)
=\varinjlim_{(F\xrightarrow{a}\underline R)\in\mathcal F/\underline R}F
\cong\varinjlim_{\beta<\omega_3}y_\beta
\cong\underline R,
\]
$\operatorname{Lan}_j$ is exact and $\operatorname{Lan}_jQ_\bullet$ is a projective resolution of $I_{\mathcal P^+}$.
Since $\operatorname{Lan}_h\operatorname{Lan}_kQ_\bullet\cong\operatorname{Lan}_jQ_\bullet$, we have $\mathbf L_n\operatorname{Lan}_h(I_{\mathcal P})=0$ for $n>0$.
Thus $\operatorname{Lan}_h$ sends any projective resolution of $I_{\mathcal P}$ to one of $I_{\mathcal P^+}$; projectivity is preserved because restriction is its exact right adjoint.

Take the zero-reduced bar resolution $B^{\mathrm{red}}_\bullet\to I_{\mathcal P}$ in $\operatorname{Fun} (\mathcal P, \mathrm{Mod}_{\mathcal C})$, i.e., the quotient of usual bar complex $B_\bullet$ by its subcomplex $N_\bullet$, with terms
\[
N_n(Y)=\coprod_{\substack{X_n\to\cdots\to X_0\to Y\\\text{at least one}\\\text{arrow is zero}}}X_n, \qquad B^{\mathrm{red}}_n(Y)=
\coprod_{\substack{
X_n\xrightarrow{f_n}\cdots\xrightarrow{f_1}X_0\xrightarrow{a}Y\\
X_j\in\mathcal P,\ f_j\ne0,\ a\ne0}}
X_n.
\]
Since $H_{\geq n} N_\bullet$ is immediate,  $H_{\geq 1}B_\bullet^{\mathrm{red}}=0$ and $H_0B_\bullet^{\mathrm{red}} \cong I_{\mathcal P}$.
Each $N \to B$ is degree wise split: at degree $n$, construct the left inverse of inclusion by sending the coordinates
\[
(X_n\xrightarrow{f_n}\cdots\xrightarrow{f_1}X_0\xrightarrow aY)
\longmapsto
\begin{cases}
(X_n\xrightarrow{f_n}\cdots\xrightarrow{f_1}X_0\xrightarrow aY),&\exists j,\ f_j=0,\\
(X_n\xrightarrow{f_n}\cdots\xrightarrow{f_1}X_0\xrightarrow 0Y),&\forall j,\ f_j\ne0.
\end{cases}
\]
Hence $B_\bullet^{\mathrm{red}}$ is a well-defined projective resolution.

Note that $\operatorname{Lan}_h B_\bullet^{\mathrm{red}}$ is a projective resolution of $I_{\mathcal P^+}$.
Evaluation at $\underline R$ gives
\[
T_\bullet=(\operatorname{Lan}_hB_\bullet^{\mathrm{red}})(\underline R)\longrightarrow\underline R,
\qquad
T_n=\coprod_{\substack{X_n\xrightarrow{f_n}\cdots\xrightarrow{f_1}X_0\xrightarrow{a}\underline R\\X_j\in\mathcal P,\ f_j\ne0,\ a\ne0}}X_n.
\]

The zero-normalized cocycle $u$ defines $\widetilde u:T_4\to A$.
By construction $\widetilde u|_{(f_1,\ldots,f_4,1)} =(1\otimes_{\mathcal C}E)\circ u(f_1,f_2,f_3,f_4)$.
Naturality gives $\widetilde u|_{(f_1,\ldots,f_4,a)} =(a\otimes_{\mathcal C}E)\circ u(f_1,f_2,f_3,f_4)$.
Let $Q_n=\coprod_{\beta_0\leq\cdots\leq\beta_n}y_{\beta_0}$ be the ordinal bar resolution used to compute $H_{\mathrm{cat}}^\bullet (\omega_3, E|_{\omega_3})$.
Then one has a comparison of projective resolution
\[
\xymatrix@R=1em{
\cdots \ar[r] & Q_1 \ar[r] \ar[d]_{\psi_1}
& Q_0 \ar[r] \ar[d]_{\psi_0}
& \underline R \ar@{=}[d] \\
\cdots \ar[r] & T_1 \ar[r]
& T_0 \ar[r] & \underline R
}
\]
where $\psi_\bullet$ send index $(\beta_0\le\cdots\le\beta_n)$ to $(y_{\beta_0}\to\cdots\to y_{\beta_n} \to\underline R)$.
Fix $\sigma=(\beta_0\le\cdots\le\beta_4)$, and let
$c_j:y_{\beta_{j-1}}\to y_{\beta_j}$ and
$a_{\beta_4}:y_{\beta_4}\to\underline R$ be the canonical maps.
Write $\iota_\sigma:y_{\beta_0}\to Q_4$ for the corresponding
coproduct inclusion. Then
\[
\widetilde u\circ\psi_4\circ\iota_\sigma
=(a_{\beta_4}\otimes_{\mathcal C}E)
 \circ u(c_4,c_3,c_2,c_1).
\]
Under the canonical identifications
$y_{\beta_4}\otimes_{\mathcal C}E\cong A$
and $\underline R\otimes_{\mathcal C}E\cong A$,
the map $a_{\beta_4}\otimes_{\mathcal C}E$ is the identity.
Consequently, under the Yoneda identification
$\operatorname{Hom}_{\mathcal C}(y_{\beta_0},A)\cong A(\beta_0)$,
\[
\bigl(\widetilde u\circ\psi_4\circ\iota_\sigma\bigr)_{\beta_0}
(1_{\beta_0})
=(u|_{\omega_3})(\beta_0\le\cdots\le\beta_4).
\]
Thus $\widetilde u\circ\psi_4$ corresponds to the cochain
$u|_{\omega_3}$.

Finally we consider $\widetilde u\phi_4$ in another projective resolution $P_\bullet$ of $\underline R$, which also computes the image in $\operatorname{Ext}^4_{\mathcal C}(\underline R, A)$.
For each $n \geq 0$, let $\iota_n:P_n\to T_n$ be the canonical inclusion of the summand indexed by $P_n\xrightarrow{d_n}\cdots\xrightarrow{d_2}P_1\xrightarrow{d_1}P_0\xrightarrow{d_0}\underline R$.
Note that $\iota_\bullet$ is not a chain map: consider the summand indexed by $\cdot \xrightarrow {f_n}\cdot \cdots \cdot \xrightarrow {f_0}\cdot$, taking bar resolution gives face $\cdot \xrightarrow {f_{n-1}}\cdot \cdots \cdot \xrightarrow {f_0}\cdot$ with sign $(-1)^n$.
Hence $\phi_\bullet : P_\bullet \to T_\bullet$ is a comparison of projective resolution, where $\phi_n=(-1)^{n(n+1)/2}\iota_n$.
Since $(-1)^{4\cdot5/2}=1$, $\widetilde u\phi_4$ and $\widetilde u\psi_4$ represent the same Ext class.
Thus $\chi([u|_{\mathcal F}])=[\widetilde u\phi_4]=[(d_0\otimes_{\mathcal C}E)\circ u(d_1,d_2,d_3,d_4)]$.
\end{proof}

\begin{proposition}
\label{prop_characteristic_almost_constant}
For every
$z\in H^4_{\mathrm{cat}}(\mathcal F,\mathsf L(\mathbb E|_{\mathcal F}))$,
there is $S\subseteq I$ with $|S|<\theta$ such that
\[
(I \setminus S)\longrightarrow \operatorname{Ext}^4_{\mathcal C}(\underline R, A), \qquad i \longmapsto \chi((\operatorname{ev}_i)_*z) \qquad \text{is a constant map}.
\]
\end{proposition}

\begin{proof}
Choose a cocycle $v$ representing $\operatorname{res}(z)\in H^4_{\mathrm{cat}}\!\left(\omega_3,\mathsf L(E|_{\omega_3})\right)$.
Namely, 
\[
v=(v(\beta_0\le\cdots\le\beta_4))_{\beta_0\le\cdots\le\beta_4}
\in\prod_{\beta_0\le\cdots\le\beta_4}\mathsf L(A(\beta_0)),
\qquad\delta v=0.
\]
For each $\sigma=(\beta_0\le\cdots\le\beta_4)$, the coefficient
$v(\sigma)$ belongs to $\mathsf L(A(\beta_0))$.
Hence
\[
(I \setminus S_\sigma) \longrightarrow A(\beta_0) , \qquad i \longmapsto \operatorname{ev}_i(v(\sigma))
\]
is constant for some $S_\sigma\subseteq I$, with $|S_\sigma|<\theta$.
Put $S=\bigcup_\sigma S_\sigma$.
There are only $\aleph_3$ such strings. 
Since $\operatorname{cf}(\theta) > \aleph_3$, one has $|S| < \theta$.
Now the result follows.
\end{proof}

\subsection{The topological obstruction}\label{sec_topological_obstruction}

\begin{lemma}\label{lem_topological_model_realisation}
Every small topological triangulated category admits a fully faithful triangle functor into the homotopy category of a stable topological model category.
\end{lemma}

\begin{proof}
    Let $\mathcal T$ be a small topological triangulated category. 
    Choose a fully faithful triangle functor $\mathcal T\to\operatorname{Ho}(\mathcal M_0)$ with $\mathcal M_0$ stable.
By \cite[Remark~2.13]{BGT}, the underlying $\infty$-category of $\mathcal M_0$ is stable.
By \cite[Propositions~4.8--4.9]{BGT}, the resulting small stable $\infty$-category admits a spectral model $\mathcal A$ compatible with the stable Yoneda embedding.
The spectral Yoneda embedding into right $\mathcal A$-modules then induces a fully faithful triangle functor $\mathcal T\to D(\mathcal A)$, where $D(\mathcal A)$ denotes the derived category of $\mathcal A$-modules; see \cite[\S2.1]{BGT}.

Choose a cofibrant replacement $\mathcal A^c\to\mathcal A$ in the model category of spectral categories with fixed object set $\operatorname{Ob}(\mathcal A)$. 
Let $|\mathcal A^c|$ denote its levelwise geometric realisation.
By the module comparisons in \cite[proof of Corollary~1.2, p.~331]{SSMonoidal}, there are Quillen equivalences
\[
\operatorname{Mod}\text{-}\mathcal A
\simeq_Q
\operatorname{Mod}\text{-}\mathcal A^c
\simeq_Q
\operatorname{Mod}\text{-}|\mathcal A^c|.
\]
The last category is a stable topological model category $\mathcal M$.
Passing to homotopy categories, one has a triangle equivalence $D(\mathcal A)\simeq\operatorname{Ho}(\mathcal M)$, and hence the required fully faithful triangle functor $\mathcal T\to\operatorname{Ho}(\mathcal M)$.
\end{proof}

\begin{proposition}\label{prop_topological_universal_class}
Let $\mathcal T$ be a small topological triangulated category, with a 2-split subset $\mathcal Q\subseteq\mathcal T$.
There exists a zero-normalised cocycle $u\in Z^4_{\mathrm{cat}}\!\left(
\mathcal Q,
\operatorname{Hom}_{\mathcal T}(-,-[-2])|_{\mathcal Q^{\mathrm{op}}\times\mathcal Q}
\right)$ such that for every composable string
\[
X_4\xrightarrow{f_4}X_3\xrightarrow{f_3}X_2
\xrightarrow{f_2}X_1\xrightarrow{f_1}X_0, \qquad f_j f_{j+1}=0, \quad X_j \in \mathcal Q
\]
one has $u(f_1,f_2,f_3,f_4)\in
\langle f_1,f_2,f_3,f_4\rangle
\subseteq\operatorname{Hom}_{\mathcal T}(X_4,X_0[-2])$.
\end{proposition}

\begin{proof}
Let $\mathbb F : \mathcal T \to \operatorname{Ho}(\mathcal M)$ be such a fully faithful triangle functor, where $\mathcal M$ is a stable topological model category.
Since $\mathcal Q$ is $2$-split,
\[
\operatorname{Hom}_{\operatorname{Ho}(\mathcal M)}(\mathbb FP,\mathbb FQ[2r+1])
\cong\operatorname{Hom}_{\mathcal T}(P,Q[2r+1]),
\]
Thus $\mathbb F(\mathcal Q)$ is $2$-split.
Then \cite[Construction~5.7]{Sagave} gives a zero-normalised cochain $\widetilde u\in
C^4_{\mathrm{cat}}\!\left(\mathbb F(\mathcal Q),
\operatorname{Hom}_{\operatorname{Ho}(\mathcal M)}(-,-[-2])\right)$, which assigns each chain of $4$ composable morphisms $X_4\xrightarrow{g_4}X_3\xrightarrow{g_3}X_2
\xrightarrow{g_2}X_1\xrightarrow{g_1}X_0$ to a morphism $\widetilde u(g_1,g_2,g_3,g_4)
\in\operatorname{Hom}_{\operatorname{Ho}(\mathcal M)}(X_4,X_0[-2])$.
By \cite[Lemma~5.8]{Sagave}, $\delta\widetilde u=0$.
Take $\Theta_{P,Q}:\operatorname{Hom}_{\mathcal T}(P,Q[-2]) \xrightarrow{\sim} \operatorname{Hom}_{\operatorname{Ho}(\mathcal M)}(\mathbb FP,(\mathbb FQ)[-2])$.
For each composable string of morphisms $P_4\xrightarrow{f_4}P_3\xrightarrow{f_3}P_2\xrightarrow{f_2}P_1\xrightarrow{f_1}P_0$, set
\[
u(f_1,f_2,f_3,f_4)
:=\Theta_{P_4,P_0}^{-1}
\bigl(\widetilde u(\mathbb Ff_1,\mathbb Ff_2,\mathbb Ff_3,\mathbb Ff_4)\bigr).
\]
By naturality, for $a:P'\to P$, $b:Q\to Q'$, and $x:P\to Q[-2]$, one has
\[
\Theta_{P',Q'}\bigl(b[-2]\circ x\circ a\bigr)
=(\mathbb Fb)[-2]\circ\Theta_{P,Q}(x)\circ\mathbb Fa.
\]
Put $u = \Theta^{-1}(\mathbb F^\ast \widetilde u)$, then one has $(\delta\widetilde u)=0$, and $u$ is zero-normalized.

Since $f_j f_{j+1}=0$, \cite[Proposition~6.1]{Sagave} gives $\widetilde u(\mathbb Ff_1,\ldots,\mathbb Ff_4)
\in\langle\mathbb Ff_1,\ldots,\mathbb Ff_4\rangle$.
Then Remark~\ref{rem_toda_functoriality} gives $u(f_1,f_2,f_3,f_4)\in\langle f_1,f_2,f_3,f_4\rangle$.
\end{proof}

Let $\mathcal P$ be the full additive category generated by representable $\mathcal C$-modules and the $P_p$.
Define the diagonal functor
\[
\Delta:\mathcal F\longrightarrow
\mathcal T_{\varnothing}
=\operatorname{Perf}(\Lambda_-(I))
\longrightarrow\mathcal T.
\]
It sends $y_\alpha$ to $\mathcal{H}om_{\Lambda_-(I)}(-,\alpha)$, $f:y_\alpha \to y_\beta$ to a constant map $I \to \operatorname{Hom}_{\mathcal C}(\alpha, \beta), \quad i \longmapsto f$.

For $i\in I$, let $\kappa_i:\mathcal P\to\mathcal T$ be the composite
\[
\kappa_i:\mathcal P
\longrightarrow\mathcal V_{\operatorname{Sign}(i)}
\xrightarrow{X\mapsto(X,0)}
\mathcal T_{\{i\}}
=\mathcal V_{\operatorname{Sign}(i)}\times
\operatorname{Perf}(\Lambda_-(I\setminus\{i\}))
\longrightarrow\mathcal T.
\]

Let $\mathcal Q$ be the small full additive closure in $\mathcal T$ of the
objects of $\operatorname{im}\Delta$ and all $\operatorname{im}\kappa_i$.

\begin{lemma}
    $\mathcal Q$ is $2$-split.
    For $P, Q \in \mathcal F$, one has $\operatorname{Hom}_{\mathcal T}(\Delta P,\Delta Q[-2])
\cong\mathsf L(\mathbb E|_{\mathcal F}(P,Q))$.
\end{lemma}

\begin{proof}
    For representable $P=y_\alpha$ and $Q=y_\beta$, the definition of $\Lambda_-(I)$ gives
\[
\operatorname{Hom}_{\mathcal T}
(\Delta y_\alpha,\Delta y_\beta[-2])
\cong\mathsf L(E(\alpha,\beta)).
\]
The general case follows by taking finite coproducts in both variables.
\end{proof}

\begin{lemma}
\label{lem_core_branch_coefficients}
For each $i\in I$, there are natural transformations
\[
\kappa_i|_{\mathcal F}\xrightarrow{\mathrm{in}_i}\Delta
\xrightarrow{\mathrm{pr}_i}\kappa_i|_{\mathcal F},
\qquad \mathrm{pr}_i\mathrm{in}_i=\mathrm{id}_{\kappa_i|_{\mathcal F}}.
\]
For $P,Q\in\mathcal F$, evaluation at $i$ corresponds to the map
\[
\operatorname{Hom}_{\mathcal T}(\Delta P,\Delta Q[-2])
\longrightarrow
\operatorname{Hom}_{\mathcal T}(\kappa_iP,\kappa_iQ[-2]),
\qquad
 x\longmapsto\mathrm{pr}_{i,Q}[-2]\circ x\circ\mathrm{in}_{i,P}.
\]
\end{lemma}

\begin{proof}
    For $P=\coprod_{a=1}^{m}y_{\alpha_a}$ and $Q=\coprod_{b=1}^{n}y_{\beta_b}$, one has
    \[ 
    \operatorname{Hom}_{\mathcal T}(\Delta P,\Delta Q[-2]) \cong\prod_{a,b}\mathsf L(E(\alpha_a,\beta_b))\cong\mathsf L\!\left(\prod_{a,b}E(\alpha_a,\beta_b)\right).
\]

Under the displayed identification, write a morphism
$x:\Delta P\to\Delta Q[-2]$ as a matrix $(x_{ba})$, where
$x_{ba}\in\mathsf L(E(\alpha_a,\beta_b))$.
Evaluation at $i$ sends this matrix to $\operatorname{ev}_i(x)=(x_{ba}(i))_{b,a}$.

On the other hand, under the decomposition $\mathcal T_{\{i\}}
=\mathcal V_{\operatorname{Sign}(i)}
\times\operatorname{Perf}(\Lambda_-(I\setminus\{i\}))$, the first component of $x$ is precisely $(x_{ba}(i))_{b,a}$.
Hence
\[
(\mathrm{pr}_{i})_Q[-2]\circ x\circ(\mathrm{in}_{i})_P
=(x_{ba}(i))_{b,a}
=\operatorname{ev}_i(x). \qedhere
\]
\end{proof}

\begin{lemma}\label{lem_corner_prism_naturality}
For $\mathfrak t\in H^4_{\mathrm{cat}}\!\left(
\mathcal Q,\operatorname{Hom}_{\mathcal T}(-,-[-2])
|_{\mathcal Q^{\mathrm{op}}\times\mathcal Q}\right)$, put $\mathfrak t_i=(\operatorname{ev}_i)_*\Delta^*\mathfrak t$.
Then 
\[
\operatorname{res}(\mathfrak t_i) =\operatorname{res}\bigl((\kappa_i^*\mathfrak t)|_{\mathcal F}\bigr) \quad \text{in} \quad H^4_{\mathrm{cat}}(\omega_3,E|_{\omega_3}).
\] 
Consequently, $\chi(\mathfrak t_i)
=\chi\bigl((\kappa_i^*\mathfrak t)|_{\mathcal F}\bigr)$.
\end{lemma}

\begin{proof}
    Fix $u\in Z^4_{\mathrm{cat}}\left( \mathcal Q,\operatorname{Hom}_{\mathcal T}(-,-[-2]) |_{\mathcal Q^{\mathrm{op}}\times\mathcal Q}\right)$ with $[u] = \mathfrak t$.
    Consider a mapping sequence $P_4\xrightarrow{f_4}P_3\xrightarrow{f_3}P_2 \xrightarrow{f_2}P_1\xrightarrow{f_1}P_0$ in $\mathcal F$.
    Then 
    \[
(\Delta^*u)(f_1,\ldots,f_4) =u(\Delta f_1,\ldots,\Delta f_4)
\in\operatorname{Hom}_{\mathcal T}(\Delta P_4,\Delta P_0[-2]).
\]
Hence Lemma~\ref{lem_core_branch_coefficients} gives
\[
((\operatorname{ev}_i)_*(\Delta^*u))(f_1, \ldots, f_4)=(\mathrm{pr}_{i})_{P_0}[-2]\circ
u(\Delta f_1,\ldots,\Delta f_4)
\circ(\mathrm{in}_{i})_{P_4}.
\]
Meanwhile, a cocycle $w_i$ representing $(\kappa_i^*\mathfrak t)|_{\mathcal F}$ gives $w_i(f_1,\ldots,f_4)
=u(\kappa_i f_1,\ldots,\kappa_i f_4)$.

Both $(\operatorname{ev}_i)_*(\Delta^*u)$ and $w_i$ takes velues in $\operatorname{Hom}_{\mathcal T}(\kappa_iP_4,\kappa_iP_0[-2])$, we then show these two cocycles differ by a coboundary.
For $P_3\xrightarrow{f_3}P_2\xrightarrow{f_2}P_1\xrightarrow{f_1}P_0$ in $\mathcal F$, put
\[
\begin{aligned}
h(f_1,f_2,f_3)
=(\mathrm{pr}_{i})_{P_0}[-2]\circ\bigl(&
u((\mathrm{in}_{i})_{P_0},\kappa_i f_1,\kappa_i f_2,\kappa_i f_3) -u(\Delta f_1,(\mathrm{in}_{i})_{P_1},\kappa_i f_2,\kappa_i f_3)\\
&+u(\Delta f_1,\Delta f_2,(\mathrm{in}_{i})_{P_2},\kappa_i f_3)-u(\Delta f_1,\Delta f_2,\Delta f_3,(\mathrm{in}_{i})_{P_3})\bigr).
\end{aligned}
\]
Taking the above equations and using $\delta u =0$, one has 
\[
\begin{aligned}
(\delta h)(f_1,f_2,f_3,f_4)
={}&(\kappa_i f_1)[-2]\circ h(f_2,f_3,f_4) -h(f_1f_2,f_3,f_4)
+h(f_1,f_2f_3,f_4)\\
&-h(f_1,f_2,f_3f_4) +h(f_1,f_2,f_3)\circ\kappa_i f_4\\
={}&u(\kappa_i f_1,\ldots,\kappa_i f_4)-(\mathrm{pr}_{i})_{P_0}[-2]\circ
u(\Delta f_1,\ldots,\Delta f_4)
\circ(\mathrm{in}_{i})_{P_4}
\end{aligned}
\]
This proves $[w_i]=(\operatorname{ev}_i)_*\Delta^*[u]=(\operatorname{ev}_i)_*\Delta^*\mathfrak t\ (=\mathfrak t_i)$.

Hence $\operatorname{res}\bigl((\operatorname{ev}_i)_*\Delta^*\mathfrak t\bigr)
=\operatorname{res}\bigl((\kappa_i^*\mathfrak t)|_{\mathcal F}\bigr)$ in $H^4_{\mathrm{cat}}(\omega_3,E|_{\omega_3})$.
\end{proof}

\begin{lemma}
\label{lem_branch_characteristic_values}
Suppose that $\mathcal T$ is topological, and choose the zero-normalised
cocycle $u$ of Proposition~\ref{prop_topological_universal_class} on
$\mathcal Q$. Put $\mathfrak t=[u]$ and
$\mathfrak t_i=(\operatorname{ev}_i)_*\Delta^*\mathfrak t$.
Then
\[
\chi(\mathfrak t_i)=
\begin{cases}
0,&\operatorname{Sign}(i)=-,\\
-[c],&\operatorname{Sign}(i)=+.
\end{cases}
\]
\end{lemma}
\begin{proof}
Since $\kappa_i$ is additive, $\kappa_i^*u$ is again zero-normalised.
This cocycle is defined on $\mathcal P$, thus taking $P_4 \xrightarrow{d_4}\cdots \xrightarrow{d_1} P_0$ and applying  Lemmas~\ref{lem_corner_prism_naturality}
and~\ref{lem_finite_large_characteristic} give
\[
\chi(\mathfrak t_i)
=\chi\!\left([\,(\kappa_i^*u)|_{\mathcal F}\,]\right)=\bigl[(d_0\otimes_{\mathcal C}E)\circ
(\kappa_i^*u)(d_1,d_2,d_3,d_4)\bigr].
\]
By Proposition~\ref{prop_topological_universal_class},
\[
u(\kappa_i d_1,\kappa_i d_2,\kappa_i d_3,\kappa_i d_4)
\in
\langle\kappa_i d_1,\kappa_i d_2,\kappa_i d_3,\kappa_i d_4\rangle_{\mathcal T}
\subseteq
\operatorname{Hom}_{\mathcal T}(\kappa_iP_4,\kappa_iP_0[-2]).
\]
Then, under the identification $\operatorname{Hom}_{\mathcal T}(\kappa_iP_4,\kappa_iP_0[-2])
\cong
\operatorname{Hom}_{\mathcal V_{\operatorname{Sign}(i)}}(P_4,P_0[-2])$, 
Remark~\ref{rem_toda_functoriality} gives a canonical bijection $\langle\kappa_i d_1,\kappa_i d_2,\kappa_i d_3,\kappa_i d_4\rangle_{\mathcal T}
\cong \langle d_1,d_2,d_3,d_4\rangle_{\mathcal V_{\operatorname{Sign}(i)}}$.
Finally, Proposition~\ref{prop_corrected_toda_endpoint} now gives the asserted values.
\end{proof}

\begin{proposition}\label{prop_incidence_not_topological}
The category $\mathcal T$ is not topological.
\end{proposition}

\begin{proof}
Suppose otherwise, and choose $u$ and $\mathfrak t_i$ as in
Lemma~\ref{lem_branch_characteristic_values}.
By Lemma~\ref{lem_core_branch_coefficients},
\[
\Delta^*[u]\in
H^4_{\mathrm{cat}}(\mathcal F,\mathsf L(\mathbb E|_{\mathcal F})).
\]
Proposition~\ref{prop_characteristic_almost_constant} gives $S\subseteq I$
with $|S|<\theta$ such that $\chi(\mathfrak t_i)$ is constant
for $i\in I\setminus S$.

However, by Lemma~\ref{lem_branch_characteristic_values}, $\chi(\mathfrak t_{i}) = -[c]$ for $\operatorname{Sign}(i) = +$, and $\chi(\mathfrak t_{i}) = 0$ for $\operatorname{Sign}(i) = -$.
Here $[c] \neq 0$ by
Lemma~\ref{lem_ordinal_nonvanishing}
Both $|\operatorname{Sign}^{-1}(+)|= |\operatorname{Sign}^{-1}(-)| = \theta$, a contradiction.
\end{proof}

\section*{Statement on the use of AI}

In the first arXiv version (v1), the author proposed a construction using two perfect categories arising from two $A_\infty$-categories with the same underlying graded category. By taking alternating products, the construction produces a triangulated category which is algebraic at every finite stage but is not topological as a whole. To find a graded algebra supporting this construction, the author used GPT6.1 Sol and computer searches.

The author soon realised that the tensor construction in v1 was not essential: the obstruction could be realised directly over the ordinal category $\omega_3$, with a much simpler graded category and $A_\infty$-structures, leading to this simpler version (v2). Sections~1 (Introduction) and~3 (Main construction) of v2 are written by the author. Section~2 (Preliminaries) is written with assistance from GPT6.1 Sol; it provides background material and does not contain the main results.

We thank Yicong Zhao for his formal verification of the original version of this article in Lean.
Formal verification of the present version (in Lean) is ongoing.

\end{document}